\documentclass{article}
\usepackage{dutchcal}
\usepackage{graphicx}
\usepackage{epstopdf}
\usepackage{amsmath,amsthm}
\usepackage{amssymb}
\usepackage{pgf}
\usepackage{tikz}
\usepackage{multicol}
\usepackage{mathtools}
\usepackage{titlesec}
\usepackage{bbm}
\usepackage{comment} 
\usepackage[linkcolor=blue,citecolor=red,colorlinks=true]{hyperref}
\usetikzlibrary{arrows,automata}
\newtheorem{thm}{\bf Theorem}[section]
\newtheorem{lem}[thm]{\bf Lemma}
\newtheorem{cor}[thm]{\bf Corollary}
\newtheorem{prop}[thm]{\bf Proposition}
\newtheorem{obs}[thm]{\bf Observation}

\theoremstyle{definition}

\newtheorem{rem}[thm]{\bf Remark}

\newcommand{\be}{\begin{eqnarray}}
\newcommand{\ee}{\end{eqnarray}}
\newcommand{\bd}{\begin{displaymath}}
\newcommand{\ed}{\end{displaymath}}
\newcommand{\la}{\lambda}

\newcommand{\mbb}[1]{\mathbb{#1}}

\newcommand{\mb}[1]{\boldsymbol{#1}}
\newcommand{\mc}[1]{\mathcal{#1}}

\newcommand{\mr}[1]{\mathring{#1}}
\newcommand{\p}{\partial}

\newcommand{\ov}[1]{\overline{#1}}

\usepackage{bbm}

\newcommand{\e}{\varepsilon}

\newcommand{\sem}[1]{(G_{#1}(t))_{t\geq 0}}

\title{Network transport with nonlinear dynamics at the nodes}
\author{Jacek Banasiak\footnote{Department of Mathematics and Applied Mathematics, University of Pretoria, Pretoria, South Africa}\,\,\footnote{Institute of Mathematics,  Łódź University of Technology, Łódź, Poland}\,\,\footnote{e-mail: jacek.banasiak@up.ac.za}\;\; \& 
 Adam B\l och\footnotemark[\value{footnote}]\,\,\footnote{e-mail: adam.bloch@p.lodz.pl}
 \footnotetext{Institute of Mathematics,  Łódź University of Technology, Łódź, Poland}
 }
\date{}

\begin{document}

\maketitle
\begin{abstract}
\noindent
In this paper, we consider a network transport model in which agents moving along the edges can contribute to dynamics at nodes or bypass them. The model takes the form of a system of first-order partial differential equations coupled with a system of ordinary differential equations, and can describe a range of phenomena, from diseases in metapopulations to migratory systems with delays, to cell differentiation processes, providing a unified platform that includes network transport and delay systems as particular cases. We prove the well-posedness of the model in $ L^p$ spaces, $ 1\leq p<\infty$, study long-term asymptotics, and illustrate the theory using an SIS disease in a metapopulation consisting of several sites where the disease develops, connected by routes along which the population can migrate, as an example.
\end{abstract}

\section{Introduction}
Transport on networks has received considerable attention in the last several years, see e.g., \cite{BFNM3AS, BPChapt, batkai2017positive, KraPhysD, KSMZ, KrPuch, Mugnolo, Pucharx, Sik}. In its basic form, the network is represented by a metric directed graph (digraph)  $\mbb G$, consisting of $k$ vertices $\nu_1,\ldots,\nu_k,$ and $m$ arcs (directed edges), $e_1,\ldots,e_r,\ldots,e_m$, identified with the interval $[0,1]$, with $0$ corresponding to the tail and $1$ to the head of the arc. Denote by $I_e=\{1,\ldots,m\}$ the set of edges' indices and by $I_v=\{1,\ldots,k\}$ the set of nodes' indices. We assume that along each arc $e_r$, agents (particles, individuals) move at speed $c_r(x)>0$ from $0$ to $ 1$. We assume that the agents are continuously distributed on each $e_r$ and are described by the density $u_r(t,x), t\geq 0, x\in [0,1]$. Thus, we represent the density function on $\mbb G$ by the   vector function $(t.x)\mapsto \mb u(t,x)=(u_r(t,x))_{r\in I_e}.$ 
\begin{figure}[h]
\begin{center}
\begin{tikzpicture}[scale=0.6]
\node[draw,circle](N1) at (6,0)   {$\textcolor{black}{\nu_1}$};
\node[draw,circle](N2) at (10,0)   {$\textcolor{black}{\nu_2}$};
\node[draw,circle](N3) at (10,-4)   {$\textcolor{black}{\nu_3}$};
\node[draw,circle](N4) at (6,-4)   {$\textcolor{black}{\nu_4}$};
\draw[->,>=latex](N1) to [bend right=-20]  (N2);
\draw[->,>=latex](N1) to [bend right=-20]  (N3);
\draw[->,>=latex](N2) to [bend right=-20]  (N3);
\draw[->,>=latex](N4) to [bend right=-20]  (N1);
\draw[->,>=latex](N3) to [bend right=-20]  (N4);
\draw[->,>=latex](N4) to [bend right=-20]  (N3);
\draw[->,>=latex](N3) to [bend right=-20]  (N1);
\end{tikzpicture}
\end{center}
\caption{Digraph describing possible migrations between 4 patches}
\label{Assgraph}
\end{figure}
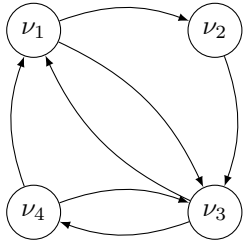
Then the flow along the edges is governed by the first-order system 
\begin{subequations}\label{tr1}
\begin{equation}
\mb u_t = -(\mc C(x)\mb u)_x, \quad x\in (0,1), t>0,
\label{tr1a}
\end{equation}
where
\begin{equation}\mc C(x) = \operatorname{diag}\,\mb c(x) = \left(\begin{array}{ccc}c_1(x)&\ldots&0\\
\vdots&\ddots&\vdots\\
0&\ldots&c_m(x)
\end{array}\right),\label{Cdef}
\end{equation}
and $c_r$ are absolutely continuous functions and there are $c_{\min},c_{\max}$ such that $0<c_{\min}\leq c_r(x)\leq c_{\max}<\infty, x\in [0,1],r\in I_e$, with the initial condition
\begin{equation}
\mb u(0,x)=\mr {\mb u}(x), \quad x\in (0,1).
\label{tr1b}
\end{equation}
This problem is supplemented by a boundary condition relating the outgoing flow to the incoming one, see, e.g., \cite[Section 1.1.5.1]{BaCo}. In general, it can be written as 
\begin{equation}\label{tr1c}
\mb u(t,0) = \mc G\mb u(t,1), \quad t>0,
\end{equation}
\end{subequations}
where, in the particular case of Kirchhoff's law at the nodes, $\mc G$ is related to the adjacency matrix of the line graph of $\mbb G$ (see Section \ref{sec2} for more details), \cite{BaFa}, but in age-structured cell mutation models, it can be arbitrary, \cite{BanRot, BPChapt}.

A natural setting for \eqref{tr1} is $L^1((0,1), \mbb C^m)$ as the norm of a (positive) solution gives the total population on $\mbb G,$ and then taking the values of $\mb u$ at $x=0,1,$ must be interpreted as trace operators. However, the $ L^p$ theory was developed in \cite{JBAB1}, see also \cite{Zwart} for another approach.

System \eqref{tr1a} can be converted into a constant speed problem by the change of variables
   $$\bar u_r(t,y_r) = c_r(x)u_r(t,x), \quad y_r = \frac{1}{\tau_r}\int_{0}^{x}\frac{ds}{c_r(s)} \in (0,1), \quad r\in I_e,$$
 and the new speed $\bar c_r$ on $e_r$ is given by
  \begin{equation}
 \frac{1}{\bar c_r} =  \tau_r :=\int_{0}^{1}\frac{ds}{c_r(s)},\label{taur}
 \end{equation}
 where $\tau_r$ is the traverse time of $e_r$, e.g., \cite{BPChapt}. Hence, from now on, we assume that $c_r, r\in I_e,$ does not depend on $x$.

Problem \eqref{tr1} has been extensively studied in the literature listed at the beginning of the section, and it has been generalised in various ways. The most relevant extension from this paper's viewpoint was introducing dynamic boundary condition by replacing $\mb u(t,0)$ in \eqref{tr1c} with $\p_t\mb u(t,0)$, \cite{Sik}, and its solution by introducing an artificial dynamics at the nodes that converted \eqref{tr1} into a matrix semigroup problem, as proposed in \cite{Nic}. This approach can be extended to a much larger class of network migration problems, which we shall discuss in detail below. Here, we provide a brief introduction to motivate the model. 

Suppose that, in addition to agents travelling along the edges according to \eqref{tr1a}, each site $\nu_i, i\in I_v,$ is inhabited by a resident population of agents $v_i$. The total population $\mb v = (v_1,\ldots,v_k)$ evolves according to a certain law $\mb v'= \mc A \mb v +\mc Q(\mb v)$, where we allow the population at one site to be instantaneously influenced by other sites independently of the existing edges, see e.g. \eqref{Mar} below.  Upon arrival at $\nu_i$, the incoming agents can either enter the site and interact with the resident population, or bypass it, travelling along the outgoing edges.  The outgoing edges can also be populated by agents leaving the site. We can model this scenario using the system
\begin{subequations}
    \label{mod1}
    \begin{equation}
    \mb u_t= - \mc C \mb u_x, \quad x\in (0,1), t>0,\label {mod1a}
    \end{equation}
    \begin{equation}
    \mb v_t = \mc A \mb v +\mc Q(\mb v) + \mc B\gamma_1 \mb u,\quad t>0, \label{mod1b}
    \end{equation}
    \begin{equation}
    \gamma_0 \mb u = \mc  F\mb v +\mc G\gamma_1 \mb u, \quad t>0,\label{mod1c}
    \end{equation}
    \begin{equation}
    \mb u(x,0) = \mr{\mb u}(x), \quad x\in (0,1),\label{mod1d}
    \end{equation}
    \begin{equation}
    \mb v(0) = \mb {\mr v}, \label{mod1e}
    \end{equation}
\end{subequations}
where $\mb u \in L_1((0,1),\mbb C^m)$, $[0,\infty)\ni t\mapsto \mb v(t)\in \mbb C^k,$
$\mc A = (a_{ij})_{i,j\in I_v}, \mc B = (b_{ij})_{i\in I_v,j\in I_e}, \mc F = (f_{ij})_{i\in I_e,j\in I_v}$ and $\mc G = (g_{ij})_{i,j\in I_e}$,
are constant matrices with real entries, $\mc Q:\Omega \to \mbb R^k,$ $\Omega \in \mbb R^k,$ is a sufficiently regular function, and $\gamma_0, \gamma_1 $ are trace operators defined for $\mb u \in W^1_1([0,1],\mbb C^m)$ by $
\gamma_0 \mb u = \mb u(0)$ and $\gamma_1 \mb u = \mb u(1).$
As with \eqref{tr1}, our basic space is 
$$
\mbb X_1 := L^1((0,1),\mbb C^m)\times l^1(\mbb C^k),
$$
with 
\begin{equation*}
\|(\mb u,\mb v)\|_{\mbb X_1} = \sum\limits_{j=1}^m\int_0^1|u_j(x)|dx + \sum\limits_{i=1}^k|v_i|.
\end{equation*}
 By $\mbb X_{1,+}$ we denote the positive cone of $\mbb X_1.$ We note that $\mbb X_1$ is an AL-space, so, in particular, the norm is additive on $\mbb X_+.$

We note that, in principle, $\mc C, \mc B, \mc F$ and $\mc G$ can also be nonlinear, see \eqref{Mar}. In this paper, however, we mostly focus on linear models ($\mc Q=0$), but return to \eqref{mod1} in Section \ref{secnon}.  

We note that even the linear version of \eqref{mod1} generalises previous models. It differs from the pure transport model \eqref{tr1a}, \eqref{tr1c} by the presence of the node dynamics \eqref{mod1b} and the corresponding unbounded boundary perturbation $\mc F\mb v$ in \eqref{mod1c}. It is also different from the model studied in \cite{Sik} and, in general, linear delay models (see Section \ref{secmd}), as it allows for a richer network structure with multiple delays between nodes and opens a possibility of adding additional dynamics during transit between them. It also 
contains an unbounded boundary perturbation $\mc G\gamma_1\mb u$ in \eqref{mod1c}. Such boundary perturbations can also be studied using the so-called Staffans--Weiss method, see \cite{HaddRhandi, Staff, Weiss}, as in e.g., \cite{BatKraRh}, but our approach through direct estimates of the resolvent is more elementary and, besides, the explicit resolvent formula allows us to derive additional results concerning the time asymptotics of the problem.

The paper is organised as follows. In Section \ref{sec2}, we recall graph-theoretical tools used in the paper. Section \ref{sec3} is devoted to two concrete applications that can be cast in the form \eqref{mod1}: a migration model and a cell differentiation model. For the former, we determine conditions under which \eqref{mod1} can be written as a system of ordinary differential equations with delay.  In Section \ref{exist}, we prove the existence of a $C_0$-semigroup $\sem{}$ solving the linear version of \eqref{mod1}, first in $\mbb X_1$ by estimating the resolvent, and next in the $L^p$ setting, by noticing that the semigroup there can be obtained as the restriction of the $\mbb X_1$ semigroup. Section \ref{sec5} is devoted to the time asymptotics of $\sem{}$. First, by extending the approach of \cite{Sik}, we show the conditions for the asymptotic stability when $\sem{}$ is positive,  and next, we generalise the result to arbitrary semigroups. In particular, we find conditions for stochasticity of $\sem{}$ when \eqref{mod1} describes a pure migration process. Finally, in Section \ref{secnon}, we provide well-posedness and stability results for the full version of \eqref{mod1}, and for illustration, we consider an SIS disease in a metapopulation described by $\mbb G,$ and find conditions for asymptotic stability of the disease-free equilibrium, using the methods developed in Section \ref{sec5}. 
\section{Graph theory toolbox}\label{sec2}
As mentioned earlier, we consider a directed graph $\mbb G$ with vertices $V=\{\nu_i\}_{i\in I_v}$ and oriented edges (arcs) $E=\{e_j\}_{j\in I_e}.$ We note that we do not consider loops (arcs from a vertex to itself) and multiple arcs.  
A related concept is that of the line graph of $\mbb G$, denoted by 
$L(\mbb G)$, and defined as the digraph whose vertices are identified with the arcs of $\mbb G$, $\{z_j\}_{j\in I_e}=\{e_j\}_{j\in I_e},$ and an arc in $L(\mbb G)$ between $z_i$ and $z_j$ occurs whenever there is a vertex $\nu_r$ of $\mbb G$ which is the head of $e_i$ and the tail of $e_j.$

We introduce two partitions of $I_e,$ 
\begin{equation*}
I_e = \{E_1^-,\ldots,E_k^-\},\quad I_e=\{E_1^+,\ldots,E_k^+\},
\end{equation*}
where, for $i\in I_v,$ $E^\pm_i$ is the set of indices of the arcs outgoing from ($-$) and incoming to ($+$) the node $\nu_i$. We note that  $E^-_i=\emptyset$ if $\nu_i$ is a sink (a vertex with no outgoing edges) and $E^+_i=\emptyset$ if $\nu_i$ is a source (a vertex with no incoming edges); if $\mbb G$ is strongly connected, then none of the sets $E_i^\pm$ is empty. We  note that, since any arc has exactly one tail and one head, 
\begin{equation}
E^\pm_i\cap E^\pm_j =\emptyset, \quad i,j\in I_v, i\neq j.
\label{Eij}
\end{equation}

We recall that the incoming incidence matrix $\Phi^+ = (\phi^+_{ij})_{ i\in I_v,j\in I_e}$ and  the outgoing incidence matrix $\Phi^-=(\phi^-_{ij})_{i\in I_e,j\in I_v}$ of $\mathbb{G}$ are defined by 
$$\phi_{ij}^+ = \left\{\begin{array}{ll} 1,&\text{if}\;j\in E^+_i\\
0,&\text{otherwise}
\end{array}\right.,
\quad \phi_{ij}^- = \left\{\begin{array}{ll} 1,&\text{if}\;i\in E^-_j\\
0,&\text{otherwise}
\end{array}\right.
.$$
By \eqref{Eij}, there is exactly one positive entry, equal to 1, in each column of $\Phi^+.$ Similarly, each row of $\Phi^-$ can contain at most one non-zero entry.

We observe that 
$$\ov{\mc N}=(\ov{\mc n}_{ij})_{i,j\in I_e} = \Phi^-\Phi^+,$$
is the adjacency matrix of $L(\mbb G),$ and $\overline{\mc M}= \Phi^+\Phi^-$ is the adjacency matrix of $\mbb G$. 
We note that not every nonnegative matrix can be the adjacency matrix of a line graph, \cite{BaFa, bein, Bang}. 
In particular, we shall need the following property.
\begin{lem}\label{lemn}
We have $\ov{\mc n}_{ij}=1$ if and only if there is $r\in I_v$ such that $e_i$ is outgoing from and $e_j$ incoming to $v_r$. In this case,
\begin{equation*}
\ov{\mc n}_{ij}=\phi^-_{ir}\phi^+_{rj}.
\end{equation*}
\end{lem}
\begin{proof}
We observe that $\Phi^-$ has only one positive entry $\phi^-_{ir}$ in each row and $\Phi^+$ has only one positive entry $\phi^+_{rj}$ in each column, with $i$ corresponding to the edge $e_i$ outgoing from $\nu_r$ and $j$ corresponding to the edge $e_j$ incoming to $\nu_r$. Thus, in the product $\Phi^-\Phi^+$, we will have either zeroes or a single product of nonzero entries in $\Phi^+$ and $\Phi^-.$ \end{proof}

\section{Examples of application}\label{sec3}

\subsection{Population migrations with delays}\label{secmd}
Consider a population divided into $k$ sub-populations, residing at different geographical sites, identified with the vertices $\{\nu_i\}_{i\in I_v}$. We assume that the changes in the population only occur through individuals' migrations between the sites. Let $v_i(t), i\in I_v,$ be the size of the sub-population at $\nu_i$ and at time $t$, and define $\mb v(t)=(v_1(t),\ldots, v_k(t)).$
For each $i\neq j$, let  $\mc m_{ij}\geq0$ be the migration rate from $\nu_j$ to $\nu_i$ and let us define the total migration rate from site $j$ as
$$\mc m_{jj}:=-\sum_{\substack{i=1\\ i\neq j}}^k \mc m_{ij}, \quad j\in I_v.$$  Then, the migration matrix is given by
$$
\mc M:=\left( \mc m_{ij}\right)_{i,j\in I_v} =: \operatorname{diag} (\mc m_{ii})_{i\in I_v} + \mc M_0 =: -\mc M_d +\mc M_0.
$$
We see that it is a Kolmogorov matrix, that is, its columns sum to zero. 

In reality, the migration between different sites does not occur instantaneously, as it takes some time to travel between them.  Hence, introducing delays in the system should give a more realistic picture. Thus, we consider, for $t>0$,
\begin{equation}\label{DE}
    v'_i(t)=\mc m_{ii}v_i(t)+\sum_{\substack{j=1\\j\neq i}}^k \mc m_{ij}v_j(t-\tau_{ij}),\quad i\in I_v,
\end{equation}
 where $\tau_{ij}$ is the time needed to travel from $\nu_j$ to $\nu_i$. 
 
 On the other hand, the migration process, taking into account the time needed to travel between the sites, can be described by \eqref{mod1}. Here, $c_j = \frac{1}{\tau_j} = \frac{1}{\tau_{rs}}, j \in E_s^-\cap E_r^+,$ that is, $\tau_j$ is the time needed to travel along the arc $e_j$ from $\nu_s$ to $\nu_r$ with speed $c_j.$ We emphasise that, though in general we can consider vector dynamics at the nodes and also instantaneous interactions between them, here we consider only scalar equations for $v_i$ at each node $\nu_i, i\in I_v$, and the effect of migration along the edges at them, and hence we assume $\mc G=0$.
 
Thus, we consider \eqref{mod1} with $\mc Q=0$, arbitrary $\mc B$ and $\mc F$, and $\mc G=0.$
 For any vector function $\mb \phi(t)=(\phi_1(t),\ldots,\phi_m(t)),$ we define  
$$\mb \phi_{\mb \tau}(t) :=\left(\begin{array}{c} \phi_1(t-\tau_1)\\\vdots\\\phi_m(t-\tau_m)\end{array}\right).$$
If $\mb u$ is a solution to \eqref{mod1a}, then, for  $j\in I_e$, 
\begin{equation}u_j(t,x) = \mc u_j(t-c_j^{-1}x) =\mc u_j(t-\tau_j x)
\label{del}
\end{equation}
for some functions $\mc u_j$, that is, $$
\mb u(t,1)=\mb  u_{\mb \tau}(t,0).
$$
Since some traverse times can equal each other, we partition $\{\tau_j\}_{j\in I_{{ e}}}$  into $ l, 1\leq l \leq m,$ sets $T_r,$ 
where for each $r,$  $\tau_i=\tau_j$ if  $\tau_i,\tau_j\in T_r.$ By renumbering the edges, we assume $\tau_r \in T_r$.
Define 
$$
I_r:=\{i:\; \tau_i\in T_r\},\quad \text{and} \quad 
I_{i,j} :=\{s\in I_e:\; b_{is}f_{sj} \neq 0\}. 
$$
\begin{prop}\label{prop0}
Let $(\mb u,\mb v)$ be a classical solution to \eqref{mod1} with $\mc G=0$ and arbitrary $\mc B$ and $\mc F$. Then $\mb v$ satisfies  \eqref{DE} with some matrix $\mc M_0$ if and only if 
for any $i,j\in I_{v}$ there is $r \in \{1,\ldots,l\}$ such that $I_{i,j}\subset I_r$. 
\end{prop}
\begin{proof} We can assume $\mc A\equiv 0$ as this matrix does not influence the structure of the connections in the graph. Then, using \eqref{del}, for each $i\in I_{{v}}, $ 
\begin{equation}\label{aux}
v_i' = \sum\limits_{s=1}^m b_{is}u_s(t,1) = \sum\limits_{s=1}^m b_{is}\sum\limits_{j=1}^k f_{sj}\upsilon_j(t-\tau_s)= \sum\limits_{j=1}^k\left(\sum\limits_{s=1}^m b_{is}f_{sj}\upsilon_j(t-\tau_s)\right).
\end{equation}
The result follows by noticing that the inner sum on the right-hand side above defines an entry of $\mc M_0$ in \eqref{DE} if and only if, for any given $i,j\in I_v$,  whenever $b_{is}f_{sj} \neq 0$ for some $s$, the corresponding $\tau_{s}$s are equal. Then $\mc m_{ij} = \sum\limits_{s\in I_{i,j}} b_{is}f_{sj}.$
\end{proof}
 To interpret Proposition \ref{prop0}, we observe that whenever $b_{is}f_{sj}\neq 0$, there is an input from $\nu_j$ into $\nu_i$ and the structure of \eqref{DE} requires any such input to arrive with the same delay.
 
In what follows, we shall present a particular realisation of the structure introduced in Proposition \ref{prop0} that uses the underlying graph structure of the system. Then, the entries of $\mc B$ must be interpreted as the rates at which the population flows into a node along the arcs connecting it with other nodes, so it must have the form $\mc B = \Phi^+\mc C$.  On the other hand, $\mc F$ should have the structure of  $\Phi^-$. However, to ensure the conservativeness of the process, the outflow from $\nu_i$ must equal the total flow along the arcs outgoing from it,  and thus we have to redistribute the flow among them. For this, we assume that each  $e_j, j\in E_i^-,$ will receive a fraction $w_{ji}$ of the inflow. Clearly, in this setting, the weights $w_{ji}$ satisfy
\begin{equation}
\sum\limits_{j\in E^-_i}^{}w_{ji}=1, \quad i\in I_v,
\label{sumom}
\end{equation}
as the matter is supposed to be neither lost nor created at the nodes. This necessitates introducing the weighted outgoing incidence matrix $\Phi^-_w,$ given by
\begin{equation}
\phi_{w,ji}^- = \left\{\begin{array}{ll} w_{ji},&\text{if}\;j \in E^-_i,\\
0,&\text{otherwise.}
\end{array}\right.
\label{om1}
\end{equation}
To balance the rates of the flow, we take the boundary condition in the form 
$$
\mc C\mb u(t,0) = \Phi_w^-\mc M_d\mb v(t), \quad t>0,
$$
and hence we consider the problem 
\begin{subequations}
    \label{de1}
    \begin{equation}
    \mb u_t= - \mc C \mb u_x, \quad x\in (0,1), t>0,\label {de1a}
    \end{equation}
    \begin{equation}
    \mb v_t (t)= -\mc M_d\mb v (t)+ \Phi^+\mc C \mb u(t,1),\quad t>0, \label{de1b}
    \end{equation}
    \begin{equation}
     \mb u(t,0) = \mc C^{-1}\Phi^{-}_w\mc M_d\mb v(t) , \quad t>0,\label{de1c}
    \end{equation}
    \end{subequations}
with appropriate initial conditions. Then Proposition \ref{prop0} yields
\begin{cor}
    Given a diagonal matrix $\mc M_d\geq 0$ and nonnegative weights $\{\mb w_i\}_{i\in I_v}$, there is a matrix $\mc M_0\geq 0$ such that for any classical solution $(\mb u,\mb v)$ to \eqref{de1}, $\mb v$  satisfies \eqref{DE} with  $\mc M = -\mc M_d+\mc M_0$.
\end{cor}
\begin{proof}
To use \eqref{de1c}, we recall that in each row of $\Phi^-_w,$ there is at most one nonzero entry $w_{rs}$ giving the portion of the material in $\nu_s$ moving into the arc $e_r$ (there can be a zero-row if $\nu_s$ is a sink). Then,
\begin{equation}
\mc C^{-1} \Phi^-_w\mc M_d\mb v(t)=\left(\begin{array}{c} c_1^{-1}w_{1s_1} \mc m_{s_1s_1}v_{s_1}(t)\\\vdots\\c_m^{-1}w_{ms_m} \mc m_{s_ms_m}v_{s_m}(t)\end{array}\right)
\label{mcC}
\end{equation}
for some rearrangement, with possible repetitions,  $(s_1,\ldots,s_m)$  of $(1,\ldots,k),$ and hence 
 $$
\mb u(t,1)=\mb  u_{\mb \tau}(t,0) = \left(\begin{array}{c} c_1^{-1}w_{1s_1}\mc m_{s_1s_1}v_{s_1}(t-\tau_{1})\\\vdots\\c_m^{-1}w_{ms_m}\mc m_{s_ms_m}v_{s_m}(t-\tau_{m})\end{array}\right).
$$
Thus, the internal sum in the third term of \eqref{aux} consists of only one summand and thus the assumptions of Proposition \ref{prop0} are satisfied, so $\mb v$ satisfies \eqref{DE} with some $\mc M_0$. To identify its entries in the current context, we observe that 
\begin{equation}\label{phiC}
\Phi^+\mc C \mb u(t,1) = \Phi^+\left(\begin{array}{c} w_{1s_1}\mc m_{s_1s_1}v_{s_1}(t-\tau_{1})\\\vdots\\w_{ms_m}\mc m_{s_ms_m}v_{s_m}(t-\tau_{m})\end{array}\right) = \left(\begin{array}{c} \sum\limits_{i=1}^m \phi^+_{1i}w_{is_i}\mc m_{s_is_i}v_{s_i}(t-\tau_i)\\
\vdots\\
\sum\limits_{i=1}^m \phi^+_{ki}w_{is_i}\mc m_{s_is_i}v_{s_i}(t-\tau_i)
\end{array}\right).
\end{equation}
If $\phi^+_{ji}w_{is_i}\mc m_{s_is_i}\neq 0,$ then there is a connection from $\nu_{s_i}$ through arc $e_i$ to $\nu_j.$ Thus $\tau_i = \tau_{js_i}$ in \eqref{DE}.  Therefore, with $\phi^+_{ji}w_{is_i}\mc m_{s_is_i} =: \mc m_{js_i}, j\in I_v, i\in I_e,$ $j\neq s_i$,  and   $\mc A = -\mc M_d,
$
\eqref{de1b} is the same as \eqref{DE}.
\end{proof}
\begin{rem}
    We see that the proof  depends on the structure of $\Phi^-_w$. Assumption \eqref{om1} on the weights and the role of $\mc M_d$ plays a role in the conservativeness of the process, see Section \ref{cons}.
\end{rem}


\begin{rem}
If  we only allow one delay in \eqref{DE}, 
\begin{equation}\label{DE2}
\mb v'(t)= -\mc M_d \mb v (t) + \mc M_0 \mb v(t-\tau),
\end{equation}
then $\mc M_0$ can be factorized in an arbitrary way, say, $\mc M_0 = \mc B\mc C\mc C^{-1}\mc F,$ to arrive at \eqref{de1} with 
\begin{subequations}
    \label{de11}
    \begin{equation*}
    \mb v_t (t)= -\mc M_d \mb v (t)+ \mc B\mc C \mb u(t,1),\quad t>0, 
    \end{equation*}
    \begin{equation*}
     \mc C\mb u(t,0) = \mc F\mb v(t) , \quad t>0.
    \end{equation*}
    \end{subequations}
Indeed, then 
$$\mc B\mc C\mb u(t,1)=\mc B\mc C\mb u(t-\tau,0)= \mc B\mc F\mb v(t-\tau) = \mc M_0\mb v(t-\tau).$$
If, however, the delays vary and $\mc F$ is such that in the formula corresponding to \eqref{mcC} there appear more than one $v_{s_i}$ in some rows, then $\mb u_{\mb \tau}(t,0) =[\mc F\mb v(t)]_{\mb\tau} \neq \mc F\mb v_{\mb\tau}(t)$ and in $\mc B[\mc F\mb v(t)]_{\mb\tau}$ we have a mixture of different delays coming from the same source.  

Of particular importance is the factorization $\mc M_0 = \mc M_0\mc I,$ which, together with the change of variable $x\to -x$,   converts \eqref{DE2} into 
\begin{subequations}
    \label{de3}
    \begin{equation*}
    \mb u_t=  \mb u_x, \quad x\in (-1,0), t>0,
    \end{equation*}
    \begin{equation*}
    \mb v_t (t)= -\mc M_d \mb v (t)+ \mc M_0 \mb u(t,-1),\quad t>0, 
    \end{equation*}
   \begin{equation*}
   \mb u(t,0) = \mb v(t), \quad t>0,
   \end{equation*}
    \end{subequations}
typically used to analyse \eqref{DE2}, e.g., \cite[Chapter II]{BatPia} or \cite[Chapter 15]{batkai2017positive}. However, this system does not provide any interesting physical graph realisation of the transport in \eqref{DE2}. 
\end{rem}

\subsection{Cell maturation and differentiation model}

An essential role in the derivation of \eqref{phiC}, that is, in revealing the network structure of the transport, is played by the structure of $\Phi_w^-$. In fact, \eqref{mcC} expresses the fact that any arc can be fed from only one vertex. However, if we distance ourselves from the graph interpretation of \eqref{de1}, then, in principle, there could occur an input from several vertices into an arc. This justifies a need to use arbitrary matrices in \eqref{mod1}. As a concrete realisation of such a model (with $\mc F=0$), we can mention the Lebowitz--Rubinov--Rotenberg model of cell maturation, division and mutation, \cite{Rot}, analysed in the current framework in \cite{BanRot}.

More general models, which fit into the framework presented in \eqref{mod1}, have been developed in, e.g.,  \cite{Belair, DoumicMC}. Leaving aside the biological details, we have a reservoir of stem cells which undergo an asymmetric division with one progeny becoming a self-renewed stem cell and one a differentiating (progenitor, precursor) cell which matures and eventually becomes a mature cell. The model developed in \cite{DoumicMC} for stem cells $w(t)$, progenitor cells of density $u(t,x)$ (where $0\leq x\leq 1$ denotes the maturity of the cell) and the mature cells $v(t)$, written as \eqref{mod1}, is given as  
\begin{equation}
\label{Mar}
\begin{split}
\p_t u(t,x) + \p_x(g(x, s)u(t,x)) &= p(x, s)u(t,x) - d(x)u(t,x),\\
w'(t)&= (2a_w-1)p_w(s) w(t) - d_w w(t),\\
v'(t) &=  -\mu v(t) + g(1, s)u(t,1),\\
g(0, s)u(t, 0) &= 2(1 - a_w(s))p_w(s)w(t), 
\end{split}
\end{equation}
where $a_w$ is the fraction of the progeny cells remaining in the stem cell pool, $p_w, p$ denote the proliferation rates, and $d_w, d$ the death rates of, respectively,  stem and progenitor cells, and $\mu$ is the death rate of the mature cells, $g$ is the maturation rate. The process is regulated by a single feedback mechanism based on the assumption that there exist signalling molecules (cytokines) which regulate the differentiation or proliferation
process. The signal intensity depends on the level of mature cells, $s=s(v)$. The model can be easily extended to different types of progenitor cells, ageing mature cells, \cite{Belair}, extrinsic factors affecting the differentiation, etc.

We note a significant difference between \eqref{Mar} and the migration model of the previous section: the equations for $v$ and $w$ remain coupled even in the absence of the transport part. 

We note that the methods developed in this paper can handle \eqref{Mar}, except for the nonlinearity in $g$ and in the boundary condition at $x=0$ (the right-hand side of the first equation is a semilinear perturbation of the problem), but we shall not develop further theory for it. 

\section{Well-posedness of \eqref{mod1} }\label{exist}

In the following two sections, we set $\mc Q=0$ in \eqref{mod1}, and whenever we refer to \eqref{mod1}, we mean its linear version \eqref{mod2a}. 
\subsection{$L^1$ theory}

As mentioned above, we consider \eqref{mod1} in the form  
\begin{subequations}\label{mod2}
\begin{equation}\label{mod2a}
\partial_t\left(\begin{array}{c}\mb u\\\mb v\end{array}\right) = \left(\begin{array}{cc}-\mc C(x)\mc{diag}\,\mb{\p}_x&0\\\mc B\gamma_1&\mc A\end{array}\right)\left(\begin{array}{c}\mb u\\\mb v\end{array}\right) =: \mathbb A\left(\begin{array}{c}\mb u\\\mb v\end{array}\right),
\end{equation}
where $\mb \p_x =\underbrace{ (\p_x,\ldots,\p_x)}_{m\;\text{times}},$ in $
\mbb X_1$ on the domain 
\begin{equation}
D(\mbb A) = \left\{(\mb u,\mb v)\in W^1_1([0,1],\mbb C^m)\times\mbb C^k:\;  \gamma_0 \mb u = \mc  F\mb v +\mc G\gamma_1 \mb u\right\},\label{mod2b}
\end{equation}
\end{subequations}
and appropriate initial conditions. 

We begin with finding the resolvent operator of $\mathbb{A}$. Although $$\left(\begin{array}{cc}0&0\\0&\mc A\end{array}\right)$$ is a bounded operator and thus can be omitted for the proof of the generation theorem, we include it in the resolvent operator as it will be necessary for further considerations.

We list basic notation and abbreviations of various operators that will be used in the sequel.
For a given operator $\mc O$, we denote its spectrum by $\sigma(\mc O)$;  $\rho(\mc O)=\mbb C\setminus \sigma(\mc O)$ denotes the resolvent set of $\mc O,$ $r(\mc O)= \sup_{\la\in \sigma(\mc O)}|\la|$ its spectral radius, $s(\mc O)=\sup_{\la\in \sigma(\mc O)}\Re\la$ its spectral bound, and $R(\la,\mc O)=(\la I-\mc A)^{-1}, \la \in \rho(\mc O),$ its resolvent.  If $\mc O$ is a matrix, we say that it is a Metzler matrix if its off-diagonal entries are nonnegative and that it is Metzler stable if $s(\mc O)<0.$ If $\mc O$ is a matrix, $\mc O_j$ denotes its $j$th row. Further, 
\begin{subequations}\label{ops}
\begin{equation}
\mb{E}_{\lambda}(s,t) = \operatorname{diag}\left(\mb e_i(s,t)\right)_{1\leq i\leq m} = \operatorname{diag}\;\left(e^{-\frac{\lambda}{c_i}(t-s)}\right)_{1\leq i\leq m}, \quad \mb E_\la:= \mb E_\la(0,1), 
\label{E}
\end{equation}
 \begin{equation}
 \mc H_\la = \mc FR(\lambda,\mc A)\mc B +\mc G
 \label{Hl}
 \end{equation}
\begin{equation}
[R_0(\la)\mb f](x) = \mathcal{C}^{-1}\int_{0}^{x}\mb{E}_{\lambda}(s,x)\mb{f}(s)ds 
\label{R0}
\end{equation}
\begin{equation}
\mc K_\la = (\mc I - \mc G\mb{E}_{\lambda})^{-1}, \label{Kl}
\end{equation}
\begin{equation}
\mc M_\la = R(1, \mc H_\la\mb E_\la)= \left(\mc I-\left(\mc FR(\la, \mc A)\mc B+\mc G\right)\mb E_{\lambda}\right)^{-1},
\label{Ml}
\end{equation}
\begin{equation}
\mc L_\la = \mc B\mb{E}_{\lambda}\mc K_\la\mc F,\label{Ll}\end{equation}
\begin{equation}
\mc R_\la=R(0, \mc L_\la+\mc A-\la\mc I) = (\la I-\mc A-\mc L_\la)^{-1}.
\label{Rl}
\end{equation}
 \end{subequations}
 We observe that $\mc R_\la$ in \eqref{Rl} is not a resolvent as the parameter $\la$ occurs also in the operator. Precisely, as stated above, it is the resolvent at 0 of $\mc L_\la+\mc A-\la\mc I$ with a fixed $\la.$
\begin{prop}\label{propres}
    The resolvent $R(\lambda,\mathbb{A})$ of the operator $\mathbb{A}$ is given by
    $$
\left[R(\la, \mbb A)\left(\begin{array}{c}\mb f\\\mb g\end{array}\right)\right](x)=\left(\begin{array}{c}\mb u(x)\\\mb v\end{array}\right),
$$
where
    \begin{subequations}\label{res11}
    \begin{equation}
    \begin{split}
\mb u(x)&=\mb E_{\lambda}(0,x)\mc M_\la \left(\mc FR(\lambda,\mc A)\mb g + \mc H_{\lambda}
  [R_0(\la)\mb{f}](1)\right)+[R_0(\la)\mb{f}](x),\end{split}
  \label{res1aa}  
  \end{equation}
  \begin{equation}\begin{split}
    \mb v&=R(\lambda,\mc A)\left(\mb g+\mc B\mb E_{\lambda}\mc M_\la \left(\mc FR(\lambda,\mc A)\mb g + \mc H_{\lambda}
  [R_0(\la)\mb{f}](1)\right)+\mc B[R_0(\la)\mb f](1)\right),
\end{split}
\label{res1bb} \end{equation}
\end{subequations}
or, equivalently, by 
\begin{subequations}
\label{res22}
\begin{equation}
  \begin{split}
\mb u(x)&= \mb E_\la(0,x)\left(\mc K_\la\left(\mc F \mc R_\la\mb g +  
\left(\mc F \mc R_\la\left(\mc B\left(\mb{E}_{\lambda}\mc K_\la \mc G +\mc I\right)\right) + \mc G\right) [R_0(\la)\mb{f}](1)\right)\right)\\
&\phantom{x}+ [R_0(\la)\mb f](x),
\end{split}\label{res2a}
\end{equation}
\begin{equation}
\mb v = \mc R_\la\left(\mb g + \mc B\left(\mb{E}_{\lambda}\mc K_\la \mc G +\mc I\right)[R_0(\la)\mb{f}](1)\right),
\label{res2b}
\end{equation}
\end{subequations}
for sufficiently large $\Re\lambda $.
\end{prop}

\begin{proof}
The resolvent equation
 takes the form
\begin{subequations}\label{res1}
\begin{equation}
\lambda \mb u(x)+\mc C\mb u'(x)=\mb f(x),\quad x\in (0,1), 
\label{res1a}
\end{equation}
\begin{equation}
    \la \mb v -\mc A\mb v= \mb g + \mc B\gamma_1 \mb u\label {res1b}.
\end{equation}
\end{subequations}
Using the variation of constants formula and \eqref{ops}, the solution to \eqref{res1a} is given by
\be \label{rozwrez}
\label{u}
\mb{u}(x)=\mb{E}_{\lambda}(0,x)\mb y+\mathcal{C}^{-1}\int_{0}^{x}\mb{E}_{\lambda}(s,x)\mb{f}(s)ds = \mb{E}_{\lambda}(0,x)\mb y+[R_0(\la)\mb f](x),
\ee
for some $\mb {y}=(y_{j})_{j\in I_e}$. Then, $\mb v$ and $\mb y$ can be determined from the following system of linear equations: 
\begin{equation}\label{forv2}
\begin{split}
\la \mb v-\mc A\mb v -\mc B\mb{E}_{\lambda}\mb y &=\mb g+\mc B[R_0(\la)\mb f](1)\\
  -\mc F\mb v +(\mc I- \mc G\mb{E}_{\lambda})\mb y &=\mc G [R_0(\la)\mb f](1).
\end{split}
\end{equation}
We provide two ways of solving \eqref{forv2}, which depend on the spectral properties of the involved operators and are equivalent for large $\Re\la.$

For sufficiently large $\Re\lambda$, $\la\notin\sigma(\mc A)$ and thus from the first equation of \eqref{forv2} we have
\begin{equation}\label{formulaforv}
\begin{split}
\mb v&=R(\la, \mc A)(\mb g +\mc B\gamma_1\mb u)=R(\la, \mc A)\left(\mb g +\mc B\left(\mb{E}_{\lambda}\mb y+[R_0(\la)\mb f](1)\right)\right).
\end{split}
\end{equation}
Inserting it into the second equation of \eqref{forv2} and rearranging, we get 
\begin{equation}
\left(\mc I-\mc H_\la\mb E_{\lambda}\right)\mb y=\mc FR(\la, \mc A)\mb g
+\left(\mc FR(\la, \mc A)\mc B+\mc G\right)\mathcal{C}^{-1}\int_{0}^{1}\!\!\mb{E}_{\lambda}(s,1)\mb{f}(s)ds. \label{formulafory}
\end{equation}
We observe that, by \eqref{Hl}, taking $\Re\lambda$ sufficiently large, we can achieve 
\begin{equation}
r(\mc H_\la \mb E_{\lambda})<1,
\label{rr1}
\end{equation}
and thus express the inverse as the Neumann series.  Hence, we have
$$
 \mb y = \mc M_\la\left(\mc FR(\la, \mc A)\mb g + \mc H_\la
  [R_0(\la)\mb{f}](1)\right) =  R(1, \mc H_\la\mb E_\la)\left(\mc FR(\la, \mc A)\mb g + \mc H_\la
  [R_0(\la)\mb{f}](1)\right),
 $$
 where, if \eqref{rr1} is satisfied, $\mc M_\la$ is given by 
\begin{equation}
\mc M_\la\mb z=\sum\limits_{r=0}^\infty \mc H_\la^r \mb E_\la^r\mb z, \quad \mb z \in \mbb C^m.
  \label{Mla}
 \end{equation}
Inserting this expression into the formulae for $\mb u$ and $\mb v$, we get \eqref{res11}.

Equivalently, \eqref{formulaforv} can be written as 
\begin{equation}\label{formulaforv1}
\begin{split}
\mb g &=(\lambda\mc I-\mc A)\mb v -\mc B\gamma_1\mb u=(\lambda\mc I-\mc A)\mb v -\mc B\left(\mb{E}_{\lambda}\mb{y}+[R_0(\la)\mb{f}](1)\right),
\end{split}
\end{equation}
and \eqref{formulafory} as
\begin{equation}\label{formulafory1}
\begin{split}
    \mb y=\mc F\mb v+\mc G\gamma_1\mb u&=\mc F \mb v +\mc G\left(\mb{E}_{\lambda}\mb y+[R_0(\la)\mb f](1)\right).
\end{split}
\end{equation}
For large $\la$, $\mc I - \mc G\mb{E}_{\lambda}$ is invertible, so, solving \eqref{formulafory1} for $\mb y,$ substituting into \eqref{formulaforv1}  and grouping there terms containing the unknown $\mb v$ 
\begin{equation}
    \begin{split}
&\left(\lambda\mc I-\mc A -\mc B\mb{E}_{\lambda}\mc K_\la\mc F\right)\mb v = 
\mb g + \mc B\left(\mb{E}_{\lambda}\mc K_\la \mc G +\mc I\right)[R_0(\la)\mb{f}](1).
\end{split}\label{Rleq}
\end{equation}
Then, for large $\la$, $\lambda\mc I-\mc A -\mc L_\la$, see \eqref{Ll}, is invertible and 
$$
\mb v = \mc R_\la\left(\mb g + \mc B\left(\mb{E}_{\lambda}(0,1)\mc K_\la \mc G +\mc I\right)[R_0(\la)\mb{f}](1)\right).
$$
Therefore, 
\begin{equation*}
\begin{split}
\mb y &=  \mc K_\la\mc F \mc R_\la\left(\mb g + \mc B\left(\mb{E}_{\lambda}\mc K_\la \mc G +\mc I\right)[R_0(\la)\mb{f}](1)\right) + \mc K_\la\mc G [R_0(\la)\mb{f}](1)\\
&= \mc K_\la\left(\mc F \mc R_\la\mb g +  
\left(\mc F \mc R_\la\left(\mc B\left(\mb{E}_{\lambda}\mc K_\la \mc G +\mc I\right)\right) + \mc G\right) [R_0(\la)\mb{f}](1)\right),
\end{split}
\end{equation*}
and, substituting it into \eqref{u}, we get \eqref{res22}. 

It is easy to convince ourselves that both formulae define the resolvent. Indeed, formulae \eqref{res1a} and \eqref{res2a} give a solution in $W^1_1((0,1),\mbb C^m)$ to \eqref{res1aa} and the other terms are unique solutions to the relevant finite-dimensional linear equations.  
\end{proof}
To bring together the representations \eqref{res11} and \eqref{res22}, we observe that, if 
\begin{equation}\label{pos}
\mc B\geq 0,\mc G\geq 0,\mc F\geq 0\;\text{and}\;\mc A\;\text{is\;a\;Metzler\;matrix},
\end{equation}
a positive operator $\mc R_\la$  exists if and only if, e.g., \cite[Theorem 8.1]{banpop},
$$
s(\mc L_\la +\mc A-\la \mc I)<0,
$$
provided $\mc K_{\lambda}$ (and hence $\mc L_{\lambda}$) is positive. The condition
\begin{equation}
r(\mc G\mb E_\la)<1,
\label{Gass}
\end{equation}
for some $\la\in \mbb R$, plays an essential role in further considerations. 
\begin{prop}\label{prop2}
Let \eqref{pos} be satisfied. Then, the following conditions are equivalent:
\begin{itemize}
    \item[(i)] $r(\mc G\mb E_\la)<1$ and $s(\mc L_\la+\mc A - \la\mc I)<0$,
    \item[(ii)] $\lambda>s(\mc A)$ and $r(\mc H_\la \mb E_\la)<1$.
\end{itemize}
\end{prop}
\begin{proof}
By (i), $\mc K_\la$ exists and is nonnegative, therefore, $\mc L_\la$ is also nonnegative. Thus, by the monotonicity of the spectral bound for nonnegative perturbations, we get  $\la>s(\mc A)$. Therefore, $\mc A-\la\mc I$ is a stable Metzler matrix and $\mc L_\la +(\mc A-\la \mc I)$ represents the so-called regular split, see \cite{Varga}, and by \cite[Theorem 3.29]{Varga} or \cite[Theorem 8.5]{banpop},
we have 
\begin{equation}
s(\mc L_\la+\mc A - \la\mc I)<0\quad\text{if\;and\;only\;if}\quad r(\mc L_\la R(\la, \mc A))<1.
\label{sr1}
\end{equation}
Using the equality $r(\mc R\mc S) = r(\mc S\mc R),$ valid for any matrices, from \eqref{Ll}, 
\begin{equation}
r(\mc L_\la R(\la, \mc A))=r(\mc B\mb{E}_{\lambda}\mc K_\la\mc FR(\la, \mc A) )= r(\mc FR(\la, \mc A)\mc B\mb{E}_{\lambda}\mc K_\la ). 
\label{sr2}
\end{equation}
Consider now
$$
\mc H_\la \mb E_\la-\mc I=\mc FR(\la, \mc A)\mc B\mb{E}_{\lambda}+  \mc G\mb{E}_{\lambda}-\mc I.
$$
The first term is nonnegative and the second is a stable Metzler matrix on account of the positive invertibility of  $-(\mc G\mb{E}_{\lambda}-\mc I),$ ensured by \eqref{Gass}, see, e.g., \cite[Theorem 8.1]{banpop}.
Using again \cite[Theorem 3.29]{Varga} or \cite[Theorem 8.5]{banpop}, we see that 
$$
s(\mc H_\la \mb E_\la-\mc I)<0 \quad\text{if\;and\;only\;if}\quad r(\mc FR(\la, \mc A)\mc B\mb{E}_{\lambda}\mc K_{\lambda})<1.
$$
However, $
s(\mc H_\la \mb E_\la-\mc I)<0
$
is equivalent to $s(\mc H_\la \mb E_\la)<1,$ and, since $\mc H_\la \mb E_\la$ is nonnegative, $s(\mc H_\la \mb E_\la)=r(\mc H_\la \mb E_\la)$. This, by \eqref{sr1}, gives the second part of (ii).

On the other hand, $\la>s(\mc A)$ gives the existence of $\mc H_\la\mb E_\la$ and yields the positivity of $\mc FR(\la,\mc A)\mc B\mb E_\la$, hence,  by the monotonicity of the spectral radius, $r(\mc H_\la\mb E_\la)<1$ implies $r(\mc G\mb E_\la)<1.$ Retracing the steps, we obtain  
$$
r(\mc H_\la \mb E_\la)<1 \quad\text{if\;and\;only\;if}\quad r(\mc FR(\la, \mc A)\mc B\mb{E}_{\lambda}\mc K_{\lambda})<1,
$$
which, by \eqref{sr1} and \eqref{sr2}, proves the second part of (i).
\end{proof}
We are ready to prove the following theorem.
\begin{thm}
   For any matrices  $\mc A, \mc B,\mc F,\mc G,$ the operator $(\mathbb{A},D(\mathbb{A}))$ generates  a $C_0$-semigroup, say, $\sem{1}$, in $\mbb X_1$. If \eqref{pos} is satisfied, $\sem{1}$ is positive.
\end{thm}
\begin{proof}
As noted earlier, for the generation theorem, we can assume $\mc A=0$. First, we show that $D(\mbb A)$ is dense in $\mbb X_1.$ Indeed, define 
\begin{equation}
C_{\mc F}:= \{\mb u\in C^1([0,1],\mbb C^m):\;\mb u(1)=0, \mb u(0)=\mc F\mb v\;\text{for\;some\;}\mb v\in \mbb C^k\}.
\label{CF}
\end{equation}
Clearly, $C_{\mc F}\times\mbb C^k\subset D(\mbb A)$. Let $(\mb f,\mb v)\in \mbb X_1$ be arbitrary and define $\mb g(x) =\mb f(x)-(1-x)\mc F\mb v.$ Since $\mb g\in L^1((0,1),\mbb C^m),$ there is a sequence $(\mb \phi_n)_{n\in \mbb N}\subset C_0^\infty((0,1),\mbb C^m)$ such that $\lim\limits_{n\to \infty} \mb \phi_n = \mb g$ in $L^1((0,1),\mbb C^m)$. Then, $\mb u_n(x):=\mb \phi_n(x)+(1-x)\mc F\mb v\in C_{\mc F}$ and $\lim\limits_{n\to \infty} \mb u_n = \mb f$ in $L^1((0,1),\mbb C^m)$. Hence, $D(\mbb A)\ni (\mb u_n,\mb v)$ converges to $(\mb f,\mb v)$ in $\mbb X_1.$ 

 Then, by \eqref{res11} or \eqref{res22}, and \eqref{Mla},   for sufficiently large  $\la\in \mbb R$, $|R(\la, \mbb A)|\leq R(\la, \mbb A_+)$, where in $\mbb A_+$ all matrices are replaced by their moduli, making $\mbb A_+$ resolvent positive. Then, by induction,
$$
 \left\| R(\lambda,\mbb{A})^{n}\right\| \leq \left\| R(\lambda,\mbb{ A}_+)^{n}\right\|, \quad n\in \mbb N,
$$
hence it suffices to prove the generation for $\mbb A_+$. Thus, in what follows, we assume that $\mc B,\mc F,\mc G$ are nonnegative. 

We re-write \eqref{formulaforv} and \eqref{formulafory} as
\be\label{res2}
\mb v = \la^{-1}(\mb g +\mc B\gamma_1 \mb u) = 
\la^{-1}\left(\mb g +\mc B \left(\mb{E}_{\lambda}\mb y+ \mathcal{C}^{-1}\int_{0}^{1}\mb{E}_{\lambda}(s,1)\mb{f}(s)ds\right)\right),
\ee
\begin{equation}\label{res4}
\begin{split}
\mb y=\gamma_0 \mb u &= 
\la^{-1}\mc F\mb g + \mc H_{\lambda}
 \left(\mb{E}_{\lambda}\mb y+\mathcal{C}^{-1}\int_{0}^{1}\mb{E}_{\lambda}(s,1)\mb{f}(s)ds \right). \end{split}
\end{equation}
Then, adding the coordinates of $\gamma_0\mb u$ in \eqref{res4} multiplied by $\lambda^{-1}$, we obtain 
\be\label{suma}
\frac{1}{\lambda}\sum_{i =1}^m y_{i}=\frac{1}{\la^2} \sum_{i =1}^m \mc F_i\mb g + \frac{1}{\lambda}\sum_{i =1}^m  h_{i}\mb e_i(0,1){y}_{i}+\frac{1}{\lambda}\sum_{i=1}^m \frac{h_i}{c_i}\int_{0}^{1}\mb e_i(s,1)f_{i}(s)ds,
\ee
where $h_{i}=\sum_{j= 1}^m h_{ji}$ are the column sums of the matrix $\mc H_{\lambda}$. Next, define the norm 
$$
\left\|\mb u\right\|_{\mb c}:=\sum\limits_{i=1}^m\left\|\frac{u_i}{c_i}\right\|_{L^1((0,1),\mathbb{R})},
$$
which is equivalent to the standard one. With this norm, we have
\begin{align*}
\left\|\mb u\right\|_{\mb c}&=\sum\limits_{i=1}^m\left(\int\limits_0^1\frac{\mb e_i(0,x)}{c_i}y_i\,dx+\int\limits_0^1\int\limits_0^x\frac{\mb e_i(s,x)}{c_i}\frac{f_i(s)}{c_i}\,ds\,dx\right)\\
&=-\frac{1}{\lambda}\sum\limits_{i=1}^m\int\limits_0^1\left(\frac{d}{dx}\mb e_i(0,x)\right)y_i\,dx+\sum\limits_{i=1}^m\int\limits_0^1\frac{f_i(s)}{c_i}\int\limits_s^1-\frac{1}{\lambda}\left(\frac{d}{dx}\mb e_i(s,x)\right)\,dx\,ds\\
&=\frac{1}{\lambda}\sum\limits_{i=1}^m\left(y_i-\mb e_i(0,1)y_i-\int\limits_0^1\mb e_i(s,1)\frac{f_i(s)}{c_i}\,ds\right)+\frac{1}{\lambda}\left\|\mb f\right\|_{\mb c}.
\end{align*}
Using \eqref{suma} in the expression for $\left\|\mb u\right\|_{\mb c}$, we have 
\begin{align*}\label{uest}
\left\|\mb u\right\|_{\mb c}&=\frac{1}{\la^2} \sum_{i =1}^m [\mc F\mb g]_i + \frac{1}{\lambda}\sum_{i =1}^m  h_{i}\mb e_i(0,1){y}_{i}+\frac{1}{\lambda}\sum_{i=1}^m h_i\int_{0}^{1}\mb e_i(s,1)\frac{f_{i}(s)}{c_i}ds\\
&-\frac{1}{\lambda}\sum\limits_{i=1}^m\left(\mb e_i(0,1)y_i+\int\limits_0^1\mb e_i(s,1)\frac{f_i(s)}{c_i}\,ds\right)+\frac{1}{\lambda}\left\|\mb f\right\|_{\mb c}\\
&=\frac{1}{\la^2} \sum_{i =1}^m [\mc F\mb g]_i+\frac{1}{\lambda}\sum\limits_{i=1}^m(h_i-1)\left(\mb e_i(0,1)y_i+\int\limits_0^1\mb e_i(s,1)\frac{f_i(s)}{c_i}\,ds\right)+\frac{1}{\lambda}\left\|\mb f\right\|_{\mb c}.
\end{align*}
Using the $\ell^1$-norm in $\mathbb{R}^k$ we clearly have
$$
\left\|\mb v\right\|=\left\|\la^{-1}\left(\mb g +\mc B \left(\mb{E}_{\lambda}\mb y+ \mathcal{C}^{-1}\int_{0}^{1}\mb{E}_{\lambda}(s,1)\mb{f}(s)ds\right)\right)\right\|\geq \frac{1}{\lambda}\left\|\mb g\right\|.
$$
Hence, if there exists $\la_0$ satisfying Proposition \ref{prop2}(ii) such that $h_i\geq 1$  for any $i\in I_e,$ then
\begin{equation*}
\left\|R(\la_0, \mbb A)\left(\begin{array}{c}\mb f\\\mb g\end{array}\right)\right\|_c = \|\mb u\|_{\mb c} + \|\mb v\|\geq \frac{1}{\la_0}(\|\mb f\|_{\mb c} + \|\mb g\|).
\end{equation*}
  Thus, we have the generation by \cite[Theorem 2.5]{Arendt}. If for some $i$ we have $h_i<1$, we replace $\mc G$ by $\mc {\hat G}$ where, for each such $i$, the entries in the $i$-th column of $\mc G$ are increased so that the sum of the entries of this column, $\hat g_i,$ satisfies $\hat g_i\geq 1$. Hence, taking $\hat {\mc H}=\la^{-1}\mc {FB} +\hat {\mc G}$, we have  also  $\hat h_i\geq 1$ irrespective of $\la$ and, if necessary, we can increase $\la$ so that $r(\mc {\hat H}\mb E_{\lambda})<1$. Hence, the resolvent of such a modified $\hat{\mbb A}$ exists and, as above, 
\bd
0\leq R(\lambda,\mbb{A})\leq R(\lambda,\mbb{\hat A}),
\ed
which proves that $\mbb{A}$ generates a semigroup. Positivity follows from $\mbb{A}$ being resolvent positive, which can be seen from the representation $\eqref{res11}$: for the case $\mc A=0$ it is straightforward, while for $\mc A$ being a Metzler matrix it is guaranteed by the fact that $R(\lambda,\mc A)$ is positive.
\end{proof}

\subsection{$L^p$ theory}
Next, we shall show that the part of the operator $(\mathbb{A},D(\mathbb{A})$) is the generator also in the space $\mbb X_p =L^p((0,1),\mathbb{C}^m)\times l^p(\mathbb{C}^k)$ for $1<p<\infty$. We will use the same symbol $\|\cdot\|_p$ to denote the norms in $L^p$ and $l^p$ -- the meaning will be clear from the context. 
\begin{thm}
The restriction $\sem{p} =(G_1(t)|_{X_p})$ is a $C_0$-semigroup whose generator $\mbb A_p$ is the part of $\mbb A$ in $\mbb X_p.$
\end{thm}
\begin{proof}
We will follow the approach of \cite[Theorem 4.2]{JBAB1} and use the fact that $\mbb X_p\subset \mbb X_1$, hence $\sem{1}$ can be evaluated on $(\mb{\mr u},\mb {\mr v})\in \mbb X_p$ and, since it is an algebraic semigroup, it suffices to show that $\mbb X_p$ is invariant under its action and that it it strongly continuous. 

First, we note that the density of $C_{\mc F}$ in $\mbb X_p$ can be proved as for $\mbb X_1.$ Next, as in \textit{op.cit.}, $\sem{1}$ can be written down explicitly for small $t$. Similarly to  \eqref{taur}, 
$\tau_j(x):= \frac{x}{c_j}, j\in I_e,$
is the time needed to move from 0 to $x$ along $e_j$, with the inverse $\tau_j^{-1}(t)$ giving the distance from $0$ covered in time $t$. Suppose that $(\mb u,\mb v)$ is a classical solution to \eqref{mod1}, originating, e.g., from $C_{\mc F}\times\mbb C^k$. The general formula for the solution to the transport equation gives 
\begin{equation}
u_j(t,x)=\left\{\begin{array}{ll}
\mathring{u}_j(\tau_j^{-1}(\tau_j(x)-t)),\quad &\tau_j^{-1}(t)<x<1,\\
\varphi(t-\tau_j(x)),&0<x<\tau_j^{-1}(t),
\end{array}\right.
\label{uj}
\end{equation}
where $\varphi(t)=u_j(0,t)$ is a boundary condition. To make notation shorter, we denote $\phi_j(x,t) = \tau_j^{-1}(\tau_j(x)-t) = x-c_jt.$ Then, in particular, 
$$
u_j(t,1)=\mathring{u}_j(\phi_j(1,t))
$$
is well-defined for $0<t<\tau_j(1),$ and we have
$$
v_i(t)=\mathring{v}_i+\sum\limits_{r=1}^mb_{ir}\int\limits_0^tu_r(1,s)\,ds=\mathring{v}_i+\sum\limits_{r=1}^mb_{ir}\int\limits_0^t\mathring{u}_r(\phi_r(1,s))\,ds.
$$
Thus, using the boundary condition \eqref{mod1c}, we have
\begin{align*}
u_j(t,0)
&=\sum\limits_{r=1}^kf_{jr}\mathring{v}_r+\sum\limits_{r=1}^kf_{jr}\sum\limits_{l=1}^mb_{rl}\int\limits_0^t\mathring{u}_l(\phi_l(1,s))\,ds+\sum\limits_{l=1}^mg_{jl}\mathring{u}_l(\phi_l(1,t)),
\end{align*}
and thus, by \eqref{uj}, for $0<x<\tau_j^{-1}(t)$,
$$
u_j(t,x)=\sum\limits_{r=1}^kf_{jr}\mathring{v}_r+\sum\limits_{r=1}^kf_{jr}\sum\limits_{l=1}^mb_{rl}\!\!\!\int\limits_0^{t-\tau_j(x)}\!\!\!\mathring{u}_l(\phi_l(1,s))\,ds+\sum\limits_{l=1}^mg_{jl}\mathring{u}_l(\phi_l(1,t-\tau_j(x))).
$$
 Taking $T:=\operatorname{min}_j\tau_j(1)$, for $0\leq t\leq T,$ the semigroup solution to \eqref{mod1} is given by
\begin{align*}
&u_j(t,x)=\mathbbm{1}_{[\tau_j^{-1}(t),1)}\mathring{u}_j(\phi_j(x,t))\\
&\phantom{x}+\mathbbm{1}_{[0,\tau_j^{-1}(t))}\left(\sum\limits_{r=1}^kf_{jr}\mathring{v}_r
+\sum\limits_{r=1}^kf_{jr}\sum\limits_{l=1}^mb_{rl}\!\!\int\limits_0^{t-\tau_j(x)}\!\!\!\mathring{u}_l(\phi_l(1, s))\,ds+\sum\limits_{l=1}^mg_{jl}\mathring{u}_l(\phi_l(1,t-\tau_j(x)))\!\!\right),\\
&v_i(t)=\mathring{v}_i+\sum\limits_{r=1}^mb_{ir}\int\limits_0^t\mathring{u}_r(\phi_r(1,s))\,ds.
\end{align*}
We begin by showing that the above formula leaves $L^p$ invariant. Denote by $\|\cdot\|_{p,t}$ the $L_p$ norm on $[0,\tau_j^{-1}(t)]\subset [0,1].$  Since $\sum\limits_{r=1}^kf_{jr}\mathring{v}_r$ is a constant, we have
$$
\left(\left\|\sum\limits_{r=1}^kf_{jr}\mathring{v}_r\right\|_{p,t}\right)^p\leq \left(\sum\limits_{r=1}^k|f_{jr}|^q\right)^{\frac{p}{q}}\left\|\boldsymbol{\mathring{v}}\right\|_{p}^p = \|\mc F_j\|_q^p\|\mr {\mb v}\|_p^p,$$
and it follows that
\begin{align*}
&\left\|u_j(\cdot,t)\right\|_p^p=\int\limits_{\tau_j^{-1}(t)}^1\left|\mathring{u}_j(\phi_j(x,t))\right|^p\,dx\\
&\phantom{xx}+\!\!\int\limits_0^{\tau_j^{-1}(t)}\left|\sum\limits_{r=1}^kf_{jr}\mathring{v}_r\right.+\sum\limits_{r=1}^kf_{jr}\sum\limits_{l=1}^mb_{rl}\int\limits_0^{t-\tau_j(x)}\mathring{u}_l(\phi_l(1,s))\,ds
+\left.\sum\limits_{l=1}^mg_{jl}\mathring{u}_l(\phi_l(1,t-\tau_j(x)))\right|^p\!\!dx\\
&\leq\int\limits_{\tau_j^{-1}(t)}^1\left|\mathring{u}_j(\phi_j(x,t))\right|^p\,dx+\left(\|\mc F_j\|_q\|\mr {\mb v}\|_p+\sum\limits_{r=1}^k|f_{jr}|\sum\limits_{l=1}^m|b_{rl}|\left\|\int\limits_0^{t-\tau_j(\cdot)}\mathring{u}_l(\phi_l(1,s))\,ds\right\|_{p,t}\right.\\
&\phantom{xx}+\left.\sum\limits_{l=1}^m|g_{jl}|\left\|\mathring{u}_l(\phi_l(1, t-\tau_j(x)))\right\|_{p,t}\right)^p\\
&\leq \!\!\!\int\limits_{\tau_j^{-1}(t)}^1\!\left|\mathring{u}_j(\phi_j(x,t))\right|^p\,dx+2^{p-1}\|\mc F_j\|_q^p\|\mr {\mb v}\|_p^p+4^{p-1}\left(\sum\limits_{r=1}^k|f_{jr}|\sum\limits_{l=1}^m|b_{rl}|\left\|\int\limits_0^{t-\tau_j(\cdot)}\mathring{u}_l(\phi_l(1,s))\,ds\right\|_{p,t}\right)^p\!\!\!\\
&\phantom{xx}+4^{p-1}\left(\sum\limits_{l=1}^m|g_{jl}|\left\|\mathring{u}_l(\phi_l(1,t-\tau_j(\cdot)))\right\|_{p,t}\!\!\right)^p =:I_1+I_2+I_3+I_4,
\end{align*}
where we used  \cite[Lemma 2.2]{Adams}. Applying the discrete H\"older's inequality twice, we get
$$
I_3\leq  4^{p-1}\|\mc F_j\|_q^pk\max\limits_r\|\mc B_r\|_q^p\sum\limits_{l=1}^m\left\|\int\limits_0^{t-\tau_j(\cdot)}\mathring{u}_l(\phi_l(1,s))\,ds\right\|_{p,t}^p
$$
and 
$$I_4\leq4^{p-1}\|
\mc G_j\|_q^p\sum\limits_{l=1}^m\left\|\mathring{u}_l(\phi_l(1, t-\tau_j(\cdot)))\right\|_{p,t}^p.
$$
Next, 
\begin{align*}
   I_1=\int\limits_{\tau_j^{-1}(t)}^1\left|\mathring{u}_j(\phi_j(x,t))\right|^p\,dx= \int\limits_{c_jt}^1\left|\mathring{u}_j(x-c_jt)\right|^p\,dx &=\int\limits_0^{1-c_jt}|\mathring{u}_j(s)|^p\,ds\leq\left\|\mathring{u}_j\right\|_p^p.
\end{align*}
Further, using the H\"older's inequality again and noting that $\tau_j^{-1}\leq 1$ and $t-\tau_j(x)\leq T,$ 
\begin{align*}
    \left\|\int\limits_0^{t-\tau_j(\cdot)}\mathring{u}_l(\phi_l(1, s))\,ds\right\|_{p,t}^p&=\int\limits_0^{\tau_j^{-1}(t)}\left|\int\limits_0^{t-\tau_j(x)}\mathring{u}_l(\phi_l(1,s))\,ds\right|^p\,dx
    \leq T^{\frac{p}{q}}\int\limits_0^{1}\int\limits_0^T|\mathring{u}_l(1-c_ls)|^p\,ds\,dx\\&=\frac{T^{\frac{p}{q}}}{c_l}\int\limits_{0}^{1}\int\limits_{1-c_lT}^{1}|\mathring{u}_l(z)|^p\,dz\,dx\leq \frac{T^{\frac{p}{q}}}{c_l}\int\limits_{0}^1|\mathring{u}_l(z)|^p\,dz\leq\frac{T^{\frac{p}{q}}}{c_l}\left\|\mathring{u}_l\right\|_p^p.
\end{align*}
Finally, 
\begin{align*}
    &\left\|\mathring{u}_l(\phi_l(1,t-\tau_j(\cdot)))\right\|_{p,t}^p=\int\limits_0^{\tau_j^{-1}(t)}|\mathring{u}_l(\phi_l(1,t-\tau_j(x)))|^p\,dx =\frac{c_j}{c_l}\int\limits_{\phi_l(1, t)}^1|\mathring{u}_l(z)|^p\,ds\leq \frac{c_j}{c_l}\left\|\mathring{u}_l\right\|_p^p.
\end{align*}
Combining all the above inequalities, we find that
\begin{align*}
\left\|u_j(\cdot,t)\right\|_p^p&\leq\left\|\mathring{u}_j\right\|_p^p\!+\!2^{p-1}\|\mc F_j\|_q^p\left\|\boldsymbol{\mathring{v}}\right\|_{p}^p\!+\!4^{p-1}\|\mc F_j\|_q^pk\max\limits_r\|\mc B_r\|_q^p\sum\limits_{l=1}^m\frac{T^{\frac{p}{q}}}{c_l}\left\|\mathring{u}_l\right\|_p^p\\
&+4^{p-1}\|\mc G_j\|_q^pc_j\sum\limits_{l=1}^m\frac{\left\|\mathring{u}_l\right\|_p^p}{c_l}\leq D^u_j\left\|\boldsymbol{\mathring{u}}\right\|_p^p+D^v_j\left\|\boldsymbol{\mathring{v}}\right\|_{p}^p,
\end{align*}
where $D^u_{j}, D_j^v$ are certain constants depending only on $\mc C, T, p$ and the boundary matrices.
Similarly, for the norm of  $\mb v$, by the H\"older's inequality, 
\begin{align*}
|v_i(t)|^p&=\left|\mathring{v}_i+\sum\limits_{r=1}^mb_{ir}\int\limits_0^t\mathring{u}_r(\phi_r(1, s))\,ds\right|^p\leq2^{p-1}|\mathring{v}_i|^p+2^{p-1}\left(\sum\limits_{r=1}^m|b_{ir}|\int\limits_0^T|\mathring{u}_r(1-c_rs)|\,ds\right)^p\\
&\leq2^{p-1}|\mathring{v}_i|^p+2^{p-1}T^{\frac{p}{q}}\|\mc B_i\|_q^p\sum\limits_{r=1}^m\frac{1}{c_r}\int\limits^1_{0}|\mathring{u}_r(z)|^p\,dz\\
&=2^{p-1}|\mathring{v}_i|^p+2^{p-1}T^{\frac{p}{q}}\|\mc B_i\|\max_r\frac{1}{c_r}\left\|\boldsymbol{\mathring{u}}\right\|_p^p\leq K^v_i\left\|\boldsymbol{\mathring{v}}\right\|_{p}^p+K^u_i\left\|\boldsymbol{\mathring{u}}\right\|_p^p,
\end{align*}
where, as above, $K^u_{i}, K_i^v$ are certain constants depending only on $\mc C, T, p$ and the boundary matrices.
Finally, combining all the constants into a single constant, we get
\begin{equation}
\left\|(\boldsymbol{u}(\cdot,t),\boldsymbol{v}(t))\right\|_{\mbb X_p}
\leq D\left\|(\boldsymbol{\mathring{u}},\boldsymbol{\mathring{v}})\right\|_{\mbb X_p}, \quad t\in [0,T).\label{44}
\end{equation}
This estimate has been proven for classical solutions, but can be extended by density to $\mbb X_p.$

Next, we show strong continuity at 0. We begin with $\mb u$ and consider first $\mb {\mathring{u}}\in C_{\mc F}$. Then, 
\begin{align*}
    &\left\|u_j(\cdot,t)-u_j(\cdot,0)\right\|_p^p=\int\limits_{c_jt}^1|\mathring{u}_j(x-c_jt)-\mathring{u}_j(x)|^p\,dx\\
&+\int\limits_0^{c_jt}\left|\sum\limits_{r=1}^kf_{jr}\mathring{v}_r\right.+\sum\limits_{r=1}^kf_{jr}\sum\limits_{l=1}^mb_{rl}\int\limits_0^{t-x/c_j}\mathring{u}_l(1-c_ls)\,ds
+\left.\sum\limits_{l=1}^mg_{jl}\mathring{u}_l\left(1-c_l\left(t-\frac{x}{c_j}\right)\right)-\mathring{u}_j(x)\right|^p\!\!\!dx.\\
\end{align*}
The estimates are similar to \cite[Theorem 4.5]{JBAB1}, with the first term converging to 0 as $t\to 0^+$ on account of $\mb {\mr u}$ being differentiable, and hence uniformly continuous, see \eqref{CF}, and the second since the interval of integration shrinks to 0 and, as we proved above, the integrand is integrable with $p$th power.
 
The convergence of the $\mb v$ component of the semigroup is straightforward:
\begin{align*}
    |v_i(t)-v_i(0)|^p&=\left|\mathring{v}_i+\sum\limits_{r=1}^mb_{ir}\int\limits_0^t\mathring{u}_r(\phi_r(1, s))\,ds-\mathring{v}_i\right|^p\to 0, \quad \text{as}\;t\to 0.
\end{align*}
The result can be extended to $\mbb X_p$ by uniform boundedness \eqref{44} and the density of $C_{\mc F}$ in $\mbb X_p.$ The statement about the generator follows from \cite[Proposition in Section II 2.1]{ENshort}.
\end{proof}
\section{Long-term behaviour}\label{sec5}

Before we proceed, we mention that two special cases of \eqref{mod1} are well understood. If $\mc F=0,$ then \eqref{mod1a} and \eqref{mod1c} become decoupled from \eqref{mod1b} and can be studied separately as a pure network transport problem, e.g., \cite{BFNM3AS, BPChapt, batkai2017positive, KraPhysD, KSMZ, Matrai}. It is then known that when the traverse times $\tau_j$ are integer multiples of some reference time, the dynamics becomes finite-dimensional and is governed by iterations of the weighted adjacency matrix of the line graph (or its generalisation if the model is not graph realisable, see \cite[Section 2.2.5]{BPChapt}). On the other hand, if $\mc G=0,$ then, as discussed in Section  \ref{secmd}, \eqref{mod1} is related to a delay problem and was studied in such a context in \cite{Sik}. It turns out that in this case, the semigroup is eventually uniformly continuous (even differentiable) and compact, which allows us to draw several strong conclusions regarding its long-term behaviour, especially when it is positive. The problem we study is, however, an unbounded boundary perturbation of either case, which does not allow us to reach such strong conclusions. We shall show, however, that part of the results of \cite{Sik} remain available and extend them to problems that do not necessarily generate positive semigroups. 
\begin{prop}
The resolvent $R(\la,\mbb A)$ is a compact operator, and the spectrum of $\mbb A$ is discrete with no accumulation points and bounded from the right.
\end{prop}
\begin{proof}
By the proof of Proposition \ref{propres}, the spectrum of $\mbb A$ is determined by the solvability of the system of equations \eqref{forv2}. Thus, the spectrum consists of the roots of the equation 
\begin{equation}
det \mbb M_\la :=det \left(\begin{array}{cc} \mc A -\la \mc I&\mc B\mb{E}_{\lambda}\\
\mc F&\mc G\mb{E}_{\lambda} -\mc I\end{array}\right) = 0.\label{detA}
\end{equation}
Since the entries are entire functions of $\la$, the zeros are isolated without any accumulation point, and they are bounded from the right as \eqref{forv2} is solvable by Proposition \ref{propres} for large $\Re\la$.
\end{proof}
The main issue is to ensure that all roots of \eqref{detA}, that is, the elements of $\sigma(\mbb A)$, have negative real parts. The issue with \eqref{detA} is that it is not a typical eigenvalue problem. Thus, we first convert it to a more familiar form at the cost, however, of requiring \eqref{Gass} to be satisfied. 

In the sequel, with \eqref{Gass} satisfied, we focus our attention on the operator
$
\lambda\mc I-\mc A-\mc L_{\lambda},
$
where $\mc L_{\lambda}$ is defined in \eqref{Ll}. We observe that $
\la\mapsto \mc L_\la
$
is a matrix-valued analytic function on  $\Omega =\{\la\in \mbb C:\; r(\mc G\mb E_\la)<1\}$.   Since the spectral radius is continuous (as the eigenvalues depend continuously on the entries of the matrix and the entries of $\mc G\mb E_\la$ depend continuously on $\la$, \cite[Sections 3.1.1 \& 3.1.2]{Ortega}), $\Omega$ is open and nonempty, by \eqref{Gass}. In particular, by the monotonicity of the exponential functions, there is $\bar\la\in \{-\infty\}\cup \mbb R$ such that $(\bar\la,\infty)\subset \Omega$ ($\bar\la=-\infty$ if $\mc G=0$).
It is then straightforward to prove the following proposition.
\begin{prop}\label{prop4}
Assume that \eqref{Gass} is satisfied. Then $\la \in \sigma(\mbb A)$ if and only if 
\begin{equation}
det (\la \mc I -\mc A - \mc L_\la)=0.
\label{det2}
\end{equation}
\end{prop}
\begin{proof}
By assumption, $ \mc I- \mc G\mb{E}_{\lambda}$ is invertible, and $$\mbb M_\la\left(\begin{array}{c}\mb x_1\\\mb x_2\end{array}\right)= \left(\begin{array}{c}\mb y_1\\\mb y_2\end{array}\right)$$ can be solved by first solving 
$$
-\mc F \mb x_1 + (\mc I- \mc G\mb{E}_{\lambda})\mb x_2 = \mb y_2
$$
with respect to $\mb x_2$ and substituting the result into the first row, see \eqref{Kl} and \eqref{Rleq}. Thus, the invertibility of $\mbb M_\la$ is equivalent to the invertibility of $\la \mc I -\mc A - \mc L_\la$. In other words,  $\la \mc I -\mc A - \mc L_\la$ is not invertible if and only if \eqref{detA} is satisfied and therefore such a $\lambda$ is in the spectrum of $\mbb A.$
\end{proof}
Equation \eqref{det2} still is not a typical eigenvalue problem, and to proceed, for a fixed $\la$, we consider the invertibility of $\mu\mc I -\mc A-\mc L_\la,$ $\mu \in \mbb C.$ 
The crucial role is played by the function 
\begin{equation*}
    (\bar\la,\infty)\ni \la \mapsto s(\la):= s(\mc A+\mc L_\la)\in \{-\infty\}\cup \mbb R.
    \label{sfunct}
\end{equation*}
In a general setting, we only have the following observation. 
\begin{obs}\label{obs1}
If for some $\la_0$, $s(\la)<\la$ for all $\la \geq \la_0$, then $s(\mbb A)<\la_0$.
\end{obs}
\begin{proof}
Fix any $\la\geq \la_0$. From the definition, any $\mu$ with $\Re\mu >s(\la)$ satisfies $\mu\in  \rho(\mc A+\mc L_{\la})$. In particular, $\la \mc I-\mc A-\mc L_{\la}$ is invertible, hence $\la \notin \sigma(\mbb A).$ Thus, $[\la_0,\infty)\subset \rho(\mbb A)$. Since $\mbb A$ generates a semigroup, we have $\{\la \in \mbb C;\;\Re \la\geq \la_0\}\subset \rho(\mbb A)$ and thus $s(\mbb A)<\la_0.$
\end{proof}
We can deal with the reverse inequality if we require more of the spectrum. 
\begin{obs}\label{obs2}
If for some $\la_0$, $s(\la_0)> \la_0$ and $s(\la_0) \in \sigma(\mc A+\mc L_{\la_0}),$ then $s(\mbb A)> \la_0.$
\end{obs}
\begin{proof}
If $s(\la_0)\in \sigma(\mc A+\mc L_{\la_0})$, then also $s(\la_0)\in \sigma(\mbb A)$ and hence $s(\mbb A)\geq s(\la_0)>\la_0.$
\end{proof}
\subsection{Positive semigroups}
Precise estimates can be established if \eqref{pos} is satisfied.  First, we establish that  $R(\la,\mbb A)\geq 0$.  
\begin{prop}\label{prop5} Assume that  $\mc A$ s a Metzler matrix and $\mc B,\mc F$ and $\mc G$ are nonnegative. Then, 
$$
R(\la, \mbb A)\geq 0\;\text{if\;and\;only\;if}\; \mc A + \mc L_\la-\la\mc I\;\text{is\;Metzler\;stable}.
$$
\end{prop}
\begin{proof}
By, e.g., \cite[Theorem 8.1 \& Lemma 8.4]{banpop}, it is easy to see that if $r(\mc G\mb{E}_{\lambda})<1,$ then $ (\mc I- \mc G\mb{E}_{\lambda})^{-1}\geq 0$, and then, as in the proof of Proposition \ref{prop4}, the nonnegative invertibility of $\mbb M_\la$
is equivalent to the nonnegative invertibility of $\la \mc I -\mc A - \mc L_\la$, see \eqref{Kl} and \eqref{Rleq}. The statement about nonnegativity of $R(\la,\mbb A)$ is a consequence of \eqref{Rleq} and subsequent back substitutions. \end{proof}

\begin{rem}\label{rem3} We have a stronger statement. By the proof of \cite[Proposition 3.1]{KamSal}, if $\mbb M_\la$ is Metzler stable, then both $\mc A-\la \mc I$ and  $\mc G\mb{E}_{\lambda}-\mc I$ are Metzler stable, and  the Metzler stability of $\mbb M_\la$ is equivalent to $\mc G\mb{E}_{\lambda}-\mc I$ and  $\mc A +\mc L_\la$, see \eqref{Kl}, being Metzler stable. The former, $\mc G\mb{E}_{\lambda}$  being nonnegative, is Metzler stable if and only if $ \mc I- \mc G\mb{E}_{\lambda}$ is positively invertible, which, in turn, is equivalent to $r(\mc G\mb{E}_{\lambda})<1$. \end{rem}

With Proposition \ref{prop5}, we can adapt the considerations of \cite{Sik} to prove the following theorem.

\begin{thm}\label{mstth}
If \eqref{pos} and \eqref{Gass} are satisfied, then $\omega(\mbb A) = s(\mbb A)$ and
\begin{equation}
s(\mbb A) <0\;\text{if\;and\;only\;if}\;\bar \la<0\;\text{and}\;s(\mc A + \mc B \mc K_0 \mc F)<0.
\label{impres}
\end{equation}
\end{thm}
\begin{proof}
The equality $\omega(\mbb A) = s(\mbb A)$  follows, since we deal with positive semigroups in $L^p$ spaces. To prove the other part, we extend the relevant results of \cite{Sik} to the current setting.

First, as in \cite[Lemma 6.13]{BatPia},  we observe that the function $(\mu,\la)\mapsto R(\mu, \mc A+\mc L_\la)$ is analytic on the set $\rho = \{(\mu,\la)\in \mbb C^2:\;\mu \in \rho(\mc A+\mc L_\la)\}.$ This follows as in \textit{op.cit.} from the fact that for each $(\mu_0,\la_0)\in \rho$, the function 
$$
\Delta_{\mu,\la} = (\mu\mc I - \mc A-\mc L_\la)- (\mu_0\mc I - \mc A-\mc L_{\la_0}) = (\mu-\mu_0) -\mc B(\mb E_\la(0,1)\mc K_\la-\mb E_{\la_0}(0,1)\mc K_{\la_0})\mc F
$$
is analytic and converges to 0 as $(\mu,\la)\to (\mu_0,\la_0).$ Then, it follows that $s(\cdot)$ is nonincreasing and continuous from the left on $(\bar\la,\infty)$ and continuous at any point $\la_0$ such that $s(\la_0)$ is isolated in $\sigma(\mc A+\mc L_{\la_0})\cap \mbb R,$ \cite[Proposition 6.14]{BatPia}. In fact, in our case, $s$ is a continuous function since $\sigma(\mc A+\mc L_{\la_0})$ is discrete.  Since all the above properties used in \textit{op.cit} are available for $\mc A$ and $\mc L_\la$, we can use the abovementioned properties of $s$.

If $\bar{\lambda}<0$, then, in particular, $\mc K_0$ and hence $\mc L_0=\mc B\mc K_0\mc F$ are well defined. Assume that $s(\mc A+\mc L_0)<0$. This corresponds to the case $s(0)<0$, and the monotonicity of $s$ ensures that $s(\mc A+\mc L_{\lambda})=s(\lambda)<\lambda$ for all $\lambda \geq 0$. By Observation \ref{obs1}, $s(\mathbb{A})<0$.

To prove the converse, we observe that if $R(\la,\mbb A)\geq 0$ with $s(\mbb A)<0$, then $0\in \rho(\mbb A),$ hence 
$$
-\mbb M_0 :=\left(\begin{array}{cc} -\mc A &-\mc B\\
-\mc F&\mc I-\mc G\end{array}\right)
$$
must be positively invertible. Indeed, by the proof of Proposition \ref{propres}, the inverse gives $\mb y=\gamma_0\mb u$ and $\mb v$, which must be nonnegative. For $\mb v,$ it is straightforward as $\mb v$ is the ODE part of the solution. Also, if $\gamma_0\mb u<0,$ then, since functions in the domain of $\mbb A$ are absolutely continuous, at least one component of $\mb u$ would be negative on some interval $[0,x'].$ Thus, we have 
$$-\mbb M_0^{-1}\geq 0,$$
and $\mbb M_0$ being Metzler, must satisfy $s(\mbb M_0)<0,$ by \cite[Theorem 8.1]{banpop}. Hence, it is a Metzler stable matrix, which, by Remark \ref{rem3}, yields both $\mc A+\mc L_0$ and $\mc G -\mc I$ being Metzler stable. The former proves $s(\mc A+\mc B\mc K_0\mc F)<0$. For the latter, as in the proof of Proposition \ref{prop2}, this is equivalent to $r(\mc G)<1$, since we have $s(\mc A)<0=\lambda$. By the monotonicity of exponential functions, $r(\mc G\mb E_\la)<1$ for $\lambda\geq 0$. Hence, $\bar\la<0$ as the set $\Omega$ is open.
\end{proof}
A similar reasoning yields the following result.
\begin{cor}\label{s(A)=0}
Let \eqref{pos} be satisfied and $\bar\la<0$. Then
\begin{equation}\label{stab0}
    s(\mbb A)=0\;\text{if\;and\;only\;if}\;s(\mc A+\mc L_0) =0. 
\end{equation}
\end{cor}
\begin{proof}
   Since $\bar\la<0$, $\mc L_0$ is defined. If $s(\mc A+\mc L_0)=0$, then $0\in\sigma(\mc A+\mc L_0)=\sigma(\mathbb{A})$ and hence $s(\mathbb{A})\geq 0$. On the other hand, by monotonicity, $s(\lambda)<\lambda$ for all $\lambda >0$, that is, $\lambda\notin\sigma(\mc A+\mc L_{\lambda})$. Thus, $s(\mathbb{A})\leq 0$.

    Conversely, if $s(\mathbb{A})=0$, then, by positivity, $0\in \sigma (\mathbb{A})$, hence $0\in\sigma(\mc A+\mc L_0)$, which shows $s(\mc A+\mc L_0)\geq 0$. Further, $\varepsilon\in\rho(\mathbb{A})$ for any $\varepsilon>0$, that is,
    $$
-\mbb M_\e :=\left(\begin{array}{cc} -\mc A+\e\mc I &-\mc B\mb E_\e\\
-\mc F&\mc I-\mc G\mb E_\e\end{array}\right)
$$
is positively invertible. Thus, $\mbb M_\e$ being Metzler, must satisfy $s(\mbb M_\e)<0,$ by \cite[Theorem 8.1]{banpop}. Hence, it is a Metzler stable matrix, which, by Remark \ref{rem3}, yields $\mc A-\e\mc I+\mc L_\e$ being Metzler stable. Again, by \cite[Theorem 8.1]{banpop}, $s(A-\e\mc I+\mc L_\e)<0$ and taking the limit $\varepsilon\to 0^+$ yields $s(\mc A+\mc L_0)\leq 0$ by continuity of $s$ and $-\e\mc I +\mc L_\e\to \mc L_0$ as $\e\to 0^+$.
\end{proof}
We note a slight difference in the formulation of Corollary \ref{s(A)=0} compared to Theorem \ref{mstth}: here, the condition $s(A)=0$ does not imply $\bar{\lambda}<0$ and hence the latter must be assumed also in the necessary condition for the existence of $\mc L_0$.

\subsubsection{Application to the migration problem}\label{cons}

Let us apply \eqref{impres} in the context of Section  \ref{secmd}. We consider a more general version of \eqref{de1}, where, as before, we assume that each node $i$ is occupied by a single population $v_i, i\in I_v,$ but the population incoming to a node can interact with the residents of the node or bypass it. 

Accordingly, we write $\Phi^+ = \Phi_p^++\Phi_q^+$, where each nonzero entry of $\Phi^+$ is split as 
\begin{equation}\label{split}
\phi^+_{ij} = p_{ij}+q_{ij}, j\in E^+_i,\end{equation} with $p_{ij}+q_{ij}\leq 1$ that allows for some loss of the agents.  If we assume that there is no internal dynamics at the nodes (the population only changes due to immigration and emigration), then \eqref{de1} can be written  as 
\begin{subequations}
    \label{de2}
    \begin{equation}
    \mb u_t= -  \mc C\mb u_x, \label {de2a}
    \end{equation}
    \begin{equation}
    \mb v_t (t)=  -\mc M_d\mb v (t)+ \mc B \mb u(t,1)=-\mc M_d\mb v (t)+ \Phi^+_p\mc C \mb u(t,1), \label{de2b}
    \end{equation}
    \begin{equation}
     \mb u(t,0)=\mc F\mb v(t) + \mc G\mb u(t,1) = \mc C^{-1}\Phi^{-}_{w_1}\mc M_d\mb v(t) + \mc C^{-1}\Phi^{-}_{w_2}\Phi^+_q\mc C\mb u(t,1), \label{de2c}
    \end{equation}
    \end{subequations}
where, recall, the weights $w_1$ and $w_2$ are the weights attached to outgoing edges, see \eqref{om1}, and $\mc G$ is a weighted adjacency matrix of the line graph of $\mathbb G$.   We observe that in each column $i$  of $\Phi^-$, $\phi_{ji}^-\neq 0$ only for $j\in E^-_i.$  Moreover, the redistribution scheme can be different for the residents leaving a node and the transient travellers, which is why the weights can be different in $\mc F$ and $\mc G$, but both $w_1$ and $w_2$ satisfy  \eqref{sumom}. 
In particular,
$$
\mc K_0 = (\mc I-\mc C^{-1}\Phi_{w_2}^-\Phi^+_{q}\mc C)^{-1},
$$
and its existence is proved in the next lemma.
\begin{lem}\label{lem2}
The following conditions are equivalent:
\begin{itemize}
    \item[(i)] $\mc K_0 \geq 0$,
    \item[(ii)] $q_{ri}<1$ for at least one $i\in I_e$ (hence $i\in E^+_r$) such that   $w_{q,i}\neq 0$ for some Perron eigenvector $\mb w_q$ of $\Phi_{w_2}^-\Phi^+_{q}$.
\end{itemize}
\end{lem}
\begin{proof}
First, we observe that
$$
\mc K_0=(\mc C^{-1}\mc C-\mc C^{-1}\Phi_{w_2}^-\Phi^+_{q}\mc C)^{-1}=\mc C^{-1}(\mc I-\Phi_{w_2}^-\Phi^+_{q})^{-1}\mc C
$$
and thus $\mc K_0$ exists and is positive if and only if so is the middle factor, which in turn is equivalent to $r_q:=r(\Phi_{w_2}^-\Phi^+_{q})<1.$ 
The matrix $\Phi_{w_2}^-\Phi^+_{q}$ is a weighted adjacency matrix of  $L(\mbb G)$ and hence, by Lemma \ref{lemn} and the discussion above it, its nonzero entries are of the form $ w_{2,jr}q_{ri}$.  Using the fact that the weights sum up to 1, we obtain 
\begin{equation}
\mb 1\cdot \Phi_{w_2}^-\Phi^+_q = (q_{s_{1}1},\ldots,
q_{s_mm})\leq (1,\ldots,1) ,
\label{subper'}
\end{equation}
where $\mb 1 = (1,\ldots,1)$ and $s_j$ is the index of the vertex determined by the column $j$ (that is,  $s_j$s for $j\in E^+_r$ are equal). Now,  $\Phi_{w_2}^-\Phi^+_q$ is a nonnegative, nonzero matrix, and thus it has a dominant eigenvalue equal to its spectral radius $r_q$ and corresponding eigenvectors $\mb w_q \geq 0$ (in general, there may be multiple eigenvectors if $\mbb G$ is not strongly connected.) Multiplying \eqref{subper'} by an eigenvector $\mb w_q$, we obtain
$$
\mb 1\cdot \Phi_{w_2}^-\Phi^+_q\mb w_q = r_q\mb 1\cdot \mb w_q = (q_{s_11},\ldots,q_{s_mm})\cdot \mb w_q\leq \mb 1\cdot \mb w_q,
$$
hence
$$
r_q\leq \frac{\sum\limits_{i=1}^m q_{s_ii}w_{q,i}}{\sum\limits_{i=1}^m w_{q,i}}\leq 1.
$$
Thus, since $0\leq q_{pl}\leq 1, p\in I_v,l\in I_e$, if $q_{si}<1$ for $i$ such that $w_{q,i}\neq 0$, then $r_q<1.$ Clearly, if $q_{si}=1$ for all such $i$, then $r_q=1.$  \end{proof}

\begin{thm}\label{thm5}
Assume that Lemma \ref{lem2}(ii) holds and that $\mc m_{ii}\neq 0, i\in I_v$ (that is, there are no sinks in $\mbb G$). Then $s(\mbb A)<0$ 
if and only if  $p_{ri}+q_{ri}<1$ for such $i\in I_e$ (hence $i\in E^+_r$) that   $w_{\Phi,i}\neq 0$ for some Perron eigenvector $\mb w_{\Phi}$ of $\Phi_{w_1}^-\Phi^+_{p} +\Phi_{w_2}^-\Phi^+_{q}$.
\end{thm}
\begin{proof}
Since, by assumption, the condition (i) of Lemma \ref{lem2} is valid, we can apply Theorem \ref{mstth} and \eqref{impres}  takes the form 
\begin{equation}
s(-\mc M_d + \Phi^+_p\mc C\mc K_0\mc C^{-1}\Phi^-_{w_1}\mc M_d)=s(-\mc M_d + \Phi^+_p(\mc I-\Phi_{w_2}^-\Phi^+_{q})^{-1}\Phi^-_{w_1}\mc M_d)<0.
\label{im1}
\end{equation}
By assumption, $-\mc M_d$ is a stable Metzler matrix and $\Phi^+_p\mc C\mc K_0\mc C^{-1}\Phi^-_{w_1}\mc M_d\geq 0$, so we can again use \cite[Theorem 8.5]{banpop} to claim that \eqref{im1} is equivalent to $r(\Phi^+_p(\mc I-\Phi_{w_2}^-\Phi^+_{q})\Phi^-_{w_1})<1$ and, using the independence of the spectral radius of the order of multiplication, we must have 
\begin{equation}\label{eq10}
r(\Phi^-_{w_1}\Phi^+_{p}(\mc I-\Phi_{w_2}^-\Phi^+_{q})^{-1})<1.\end{equation}
As in the proof of Proposition \ref{prop2} (see \eqref{sr1}), we observe that 
$
\Phi^-_{w_1}\Phi^+_{p} + \Phi_{w_2}^-\Phi^+_{q} -\mc I
$
is a regular splitting and thus \eqref{eq10} is equivalent to 
$$
s(\Phi^-_{w_1}\Phi^+_{p} + \Phi_{w_2}^-\Phi^+_{q} -\mc I)<0, 
$$
that is, $s(\Phi^-_{w_1}\Phi^+_{p} + \Phi_{w_2}^-\Phi^+_{q})<1$, or -- again on account of \cite[Theorem 8.5]{banpop} -- to
$$
r_\phi:=r(\Phi_{w_1}^-\Phi^+_p+\Phi_{w_2}^-\Phi^+_q)<1.
$$
Both matrices are weighted adjacency matrices of $L(\mbb G)$, hence, as in the proof of Lemma \ref{lem2}, nonzero entries of $\Phi_{w_1}^-\Phi^+_p+\Phi_{w_2}^-\Phi^+_q$ are of the form $w_{1,jr}p_{ri}+ w_{2,jr}q_{ri}, i\in E^+_r$. Thus, similarly, 
\begin{equation}
\mb 1\cdot (\Phi_{w_1}^-\Phi^+_p+\Phi_{w_2}^-\Phi^+_q) = \left(\begin{array} {c}p_{r_11}+q_{r_11}\\
\vdots\\p_{r_mm}+q_{r_mm}\end{array}\right)\leq \left(\begin{array} {c}1\\
\vdots\\1\end{array}\right). 
\label{subper}
\end{equation}
 Let $\mb w_\Phi\geq 0$ be a Perron eigenvector  of $\Phi_{w_1}^-\Phi^+_p+\Phi_{w_2}^-\Phi^+_q$.  Multiplying \eqref{subper} by $\mb w_\Phi$, we obtain
$$
\mb 1\cdot (\Phi_{w_1}^-\Phi^+_p+\Phi_{w_2}^-\Phi^+_q)\mb w_\Phi = r_\Phi\mb 1\cdot \mb w_\Phi 
= \sum\limits_{i=1}^m (p_{r_ii}+q_{r_ii})w_{q,i}
\leq \mb 1\cdot \mb w_\Phi.
$$
Thus, as in Lemma \ref{lem2}, for $r_\Phi=s(\mbb A)<0$, it is necessary and sufficient that $p_{ri}+q_{ri}<1$ for at least one $i$ for which $w_{\Phi,i}>0$. 
\end{proof}

\begin{rem} If there is a strictly positive Perron (right) eigenvector $\mb w_q$ of the respective matrices, the assumptions of Lemma \ref{lem2} and Theorem \ref{thm5} are satisfied if $q_{ri}<1$ (respectively $p_{ri}+q_{ri}<1$) for at least one pair $r, i.$ This is the case when $\mbb G$ is strongly connected, and so is $L(\mbb G)$, hence $\Phi^-\Phi^+$ is irreducible (and this property carries over to the weighted matrices).  

In  the migration model, $q_{ri}<1$ means that some agents moving along $e_i$ join the resident population at $\nu_r,$ and $p_{ri}+q_{ri}<1$ implies that there is a loss of agents at the splitting point. 

There are reducible matrices with a strictly positive right eigenvector, but, since we know that in our case $\mb 1$ is a positive left eigenvector, such matrices must have block diagonal form, that is, $\mbb G$ is disconnected, \cite[Theorems III.6 \& 7]{Gant}. Otherwise, for $r_q<1$ (resp. $r_\Phi<1$), it suffices to ensure that $q_{ri}<1$ (resp. $p_{ri}+ q_{ri}<1$) for at least one index corresponding to strong components of the line graph, see \cite{BanLAA, BNNHM}. 

On the other hand, if  $q_{ri}<1$ (resp. $p_{ri}+ q_{ri}<1$) for all $r\in I_v,i\in E^+_{r},$ then $r_q<1$ (resp. $r_\Phi<1$) irrespective of the structure of $\mbb G.$
\end{rem}
\begin{rem}\label{rem7}
We can apply the above considerations to \eqref{stab0} using \cite[Corollary 8.6]{banpop}, which extends \cite[Theorem 8.5]{banpop}, used in the case of strict inequality in \eqref{eq10}, to the case of weak inequality. Then, from \eqref{subper}, we see that $s(\mbb A)=0$ if  $p_{rj}+q_{rj}=1$ for all $r\in I_v$ and $j\in E^+_r.$ 
\end{rem}

The statement of the remark can be strengthened as follows.

\begin{thm}
Assume in \eqref{split} we have  $p_{ij}+q_{ij}=1$, for all $j\in E^+_i, i \in I_v.$  Then, the semigroup $(G_1(t))_{t\geq 0}$ solving \eqref{de2} is stochastic in $\mbb X_1$. 
\end{thm}
\begin{proof}
First, assume that $(t,x)\mapsto (\mb u(t,x), \mb v(t))$ is a nonnegative (classical) solution to \eqref{de2}. Integrating \eqref{de2b} and denoting by $N_j:=\int_0^1u_j(t,x)dx$ the population on $e_j,$ we get 
 \begin{equation}
 \frac{d}{dt}N_j = -c_ju_j(t,1)+c_ju_j(t,0), \quad j\in I_e.
 \label{ceq1}
 \end{equation}
 Now, for a given $i\in I_v$, we sum \eqref{ceq1} over $j\in E^-_i,$ and, using \eqref{de2c}, Lemma \ref{lemn} and \eqref{sumom}, we get
 \begin{equation*}
 \begin{split}
 \frac{d}{dt}\sum\limits_{j\in E^-_i}N_j &= -\sum\limits_{j\in E^-_i}c_ju_j(t,1)+\sum\limits_{j\in E^-_i}c_ju_j(t,0)\\
 &= -\sum\limits_{j\in E^-_i}c_ju_j(t,1)+\mc m_{ii}v_i + 
 \sum\limits_{r\in E^+_i}\phi^+_{q,ir} c_ru_r(t,1), \quad i \in I_v.
\end{split}
\end{equation*}
Then, denoting by $N=N_1+\cdots+N_m$ the total population on the edges and by $V=v_1+\cdots+v_k$ the total population in the nodes, summing over $i\in I_v$,  and adding \eqref{de2b}, we obtain 
\begin{equation*}
\begin{split}
\frac{d}{dt}(N(t)+V(t)) &= -\sum\limits_{j\in I_e}c_ju_j(t,1)+\sum\limits_{i\in I_v}\mc m_{ii}v_i + 
 \sum\limits_{i\in I_v}\left(\sum\limits_{r\in E^+_i}\phi^+_{q,ir}c_r u_r(t,1)\right)\\
 &\phantom{x}-\sum\limits_{i\in I_v}\mc m_{ii}v_i + \sum\limits_{i\in I_v}\sum\limits_{r\in E_i^+}\phi_{p,ir}^+c_ru_r(t,1)\\
  &= -\sum\limits_{j\in I_e}c_ju_j(t,1) + 
 \sum\limits_{i\in I_v}\left(\sum\limits_{r\in E^+_i}\phi^+_{ir}c_ru_r(t,1)\right) =0,
 \end{split}
 \end{equation*}
 since $\phi^+_{ir}=1$ if and only if $r\in E^+_i,$ so the inner sum represents the sum of all terminal values of $c_ru_r$ entering $\nu_i$, and then we sum over all vertices so that we get all terminal values of $u_i$s as any arc has to end at some, and only one, vertex. 

 Thus, differentiable and nonnegative solutions are conservative. We know that $G_{1}(t))_{t\geq 0}$ is positive and, by the first part of the proof,
$$
\frac{d}{dt} (N(t)+V(t)) = \frac{d}{dt}\|G_{1}(t)({\mb u}, {\mb v})\|_{\mbb X_1} = 0, \quad 0\leq ({\mb u}, {\mb v})\in D(\mbb A),
$$
that is, for any $t\geq 0$, 
$$
\|G_{1}(t)({\mb u}, {\mb v})\|_{\mbb X_1} = 
\|({\mb u}, {\mb v})\|_{\mbb X_1}.
$$
For $(\mb u,\mb v)\in \mbb X_{1,+}$, we can take $D(\mbb A)_+\ni (\mb u_n,\mb v_n)\to (\mb u,\mb v), n\to\infty$, defined by, e.g., $(\mb u_n,\mb v_n) = n\int_0^{\frac{1}{n}}G_{\mbb A}(s)(\mb u,\mb v)ds,$ see \cite[Theorem 2.4 a) \& b)]{Pa} and, using the continuity of the norm,  
$$
\|G_{1}(t)({\mb u}, {\mb v})\|_{\mbb X_1} = \lim\limits_{n\to\infty}\|G_{1}(t)({\mb u}_n, {\mb v}_n)\|_{\mbb X_1} = \lim\limits_{n\to\infty}\|({\mb u}_n, {\mb v}_n)\|_{\mbb X_1} =
\|({\mb u}, {\mb v})\|_{\mbb X_1}, \quad (\mb u,\mb v)\in \mbb X_{1,+}.
$$
\end{proof}
\subsection{General semigroups}
We can partly extend Theorem \ref{mstth} to more general matrices. 
\begin{thm}\label{mstth1}
Let us decompose $\mc A = \operatorname{diag}\,{\mb a}+\mc A_0,$ $\mb a=(a_{11},\ldots,a_{kk})$ and 
denote $\mc A_+= \operatorname{diag}\,{\mb a}+|\mc A_0|$. Further, let $\bar \la\leq 0$ (where $\bar \la$ is defined for $|\mc G|$ and $\mb E_{\Re \la}$). Then $s(\mbb A)<0$, provided 
\begin{equation}
s(\mc A_+ + |\mc B||\mc K_0|| \mc F|)<0.
\label{impres3}
\end{equation}
\end{thm}
\begin{proof}
First, we provide an estimate for $R(\la,\mc A).$ Consider the system
\begin{equation*}
\mb y'=\mc A\mb y = (\operatorname{diag}\,{\mb a})\mb y+\mc A_0\mb y.
\end{equation*}
For each $i\in I_v,$ the right-hand side derivative of $|y_i(t)|$ is given by 
$$
|y_i(t)|_+' =\left\{\begin{array}{lcl}
y_i'(t)&\textit{if}&y_i(t)>0,\\
-y_i'(t)&\textit{if}&y_i(t)<0,\\
|y_i'(t)|&\textit{if}&y_i(t)=0,
\end{array}
\right.
$$
hence 
$$
|\mb y(t)|'_+ \leq (\operatorname{diag}\,{\mb a})|\mb y(t)| + |\mc A_0||\mb y(t)|.
$$
Since the right-hand side of the above is a Metzler matrix, $|\mb y(t)|\leq \mb z(t)$, where $\mb z$ satisfies 
$$
\mb z'= (\operatorname{diag}\,{\mb a})\mb z + |\mc A_0|\mb z, \quad \mb z(0)=|\mb y(0)|,
$$
see, e.g., \cite[Theorem 10 \& comment on p. 30]{Copp}; that is, $|G_{\mc A}(t)|\leq G_{\mc A_+}(t)$. By the integral representation of the resolvent, we get
\begin{equation*}
|R(\la,\mc A)|\leq R(\Re\la, \mc A_+),\quad \Re\la>s(\mc A_+).
\end{equation*}

To conclude  the proof, first, we observe that 
$$
|\mb{E}_{\lambda}(s,t)|=  \operatorname{diag}\left(|\mb e_i(s,t)|\right)_{1\leq i\leq m} = \operatorname{diag}\;\left(\left|e^{-\frac{\lambda}{c_i}(t-s)}\right|\right)_{1\leq i\leq m}= \mb{E}_{\Re\lambda}(s,t).
$$
Remembering that $\mc H_\la = \mc FR(\lambda,\mc A)\mc B +\mc G$, $$
|\mc H_\la|\leq |\mc F|R(\Re\la, \mc A_+) |\mc B|+|\mc G|=:\ov{\mc H}_{\Re\la}.
$$
Proposition \ref{prop2} and \eqref{impres3}  imply $r(\ov{\mc H}_0)<1$, hence  
$$
\left|\mc M_\la\right|:= \left|\sum\limits_{r=0}^\infty \left({\mc H}_{\lambda}\mb E_{\lambda}\right)^r \right|\leq \sum\limits_{r=0}^\infty \left( \ov{\mc H}_{\Re\lambda}\mb E_{\Re\lambda}\right)^r \leq \sum\limits_{r=0}^\infty \ov{\mc H}^r_0,
$$
where the last inequality follows from the fact that $R(\Re\la, \mc A_+) $ is a decreasing function of $\Re\la.$. Further, 
$$
\left |R_0(\la)\mb f\right| =\left|\int\limits_0^1 \mb{E}_{\lambda}(s,1)\mathcal{C}^{-1}(s)\mb{f}(s)ds\right|\leq \int\limits_0^1 \mb{E}_{\Re\lambda}(s,1)\mathcal{C}^{-1}(s)|\mb{f}(s)|ds = R_0(\Re\la)|\mb f|.
$$
Hence, for $(\mb u_\la,\mb v_\la)=R(\la, \mbb A)(\mb f,\mb g)$, we have, by \eqref{res11},
\begin{align*}
|\mb u_\la(x)|&\leq \mb E_{\Re\lambda}(0,x)|\mc M_\la| \left(|\mc F|R(\Re\lambda,\mc A_+)|\mb g| +\bar{\mc H}_{\Re\lambda}
  [R_0(\Re\la)|\mb{f}|](1)\right)+[R_0(\Re\la)|\mb{f}|](x)\\
  |\mb v_\la|&\leq
  R(\Re\la,\mc A_+)\left(|\mb g|+|\mc B|\mb E_{\Re\lambda}|\mc M_\la| \left(|\mc F|R(\Re\lambda,\mc A_+)|\mb g| \right.\right.\\
  &\left.\left.\phantom{xx}+ \bar{\mc H}_{\Re\lambda}
  [R_0(\Re\la)|\mb{f}|](1)\right)+|\mc B|[R_0(\Re\la)|\mb f|](1)\right),
 \end{align*}
that is, 
\begin{equation*}
|R(\la,\mbb A)|\leq R(\la, \mbb A_+),
\end{equation*}
where $\mbb A_+$ is the generator of the problem \eqref{mod1} with $\mc A,\mc B,\mc F,\mc G$ replaced by $\mc A_+,|\mc B|,|\mc F|$, and $|\mc G|$, respectively. Using, the Post-Widder Inversion Formula, e.g., \cite[Corollary IV 2.5]{ENshort}, we obtain
$$
|G_{\mbb A}(t)(\mb {\mr u},\mb {\mr v})|\leq G_{\mbb A_+}(t)(|\mb {\mr u}|,|\mb {\mr v}|), \quad (\mb {\mr u},\mb {\mr v})\in \mbb X_1.
$$
Since \eqref{impres3} ensures the stability of $\sem{\mbb A_+}$ by Theorem \ref{mstth}, the stability of $\sem{\mbb A}$ follows.
\end{proof}
\section{Nonlinear dynamics at the nodes}\label{secnon}
Next, we return to the full nonlinear model \eqref{mod1} with $\mc Q\neq 0.$ We focus on $\mbb X_1$ theory.
\subsection{Well-posedness of \eqref{mod1}}
 As before, we consider \eqref{mod1} as a matrix problem 
\begin{equation*}
\begin{split}
\partial_t\left(\begin{array}{c}\mb u\\\mb v\end{array}\right) &= \left(\begin{array}{cc}-\mc C\mc{diag}\,\mb{\p}_x&\mb 0\\\mc B\gamma_1&\mc A\end{array}\right)\left(\begin{array}{c}\mb u\\\mb v\end{array}\right) +  \left(\begin{array}{c} \mb 0\\ \mc Q(\mb v)\end{array}\right)=: \mbb A\left(\begin{array}{c}\mb u\\\mb v\end{array}\right) +  \mbb Q (\mb v)
\end{split}
\end{equation*}
in $\mbb X_1,$ with the linear operator $\mbb A$ defined on $D(\mbb A),$ given by \eqref{mod2b}, and $\mc Q:\Omega\to \mbb R^k$ a Lipschitz continuous function on $\Omega\subset \mbb R^k.$

Let us consider the mild formulation of \eqref{mod2}. We have 
\begin{equation}\label{mild1}
    \left(\begin{array}{c}\mb u(t)\\\mb v(t)\end{array}\right) = G_{1}(t)\left(\begin{array}{c}\mb{\mr u}\\\mb {\mr v}\end{array}\right) + \int_0^tG_{1}(t-s) \left(\begin{array}{c}\mb 0\\\mc Q(\mb v(s))\end{array}\right)ds.
\end{equation}
If we write 
$$
G_{1}(t) = \left(\begin{array}{cc}G_{11}(t)&G_{12}(t)\\G_{21}(t)&G_{22}(t)\end{array}\right),
$$
then, in particular,
$$
\left\|G_{1}(t)\left(\begin{array}{c}0\\\mb v\end{array}\right)\right\|_{\mbb X_1} = \|G_{12}(t)\mb v\|_1+\|G_{22}(t)\mb v\|_{1},
$$
and
$$
\left\|G_{1}(t)\left(\begin{array}{c}0\\\mb v\end{array}\right)-G_{1}(s)\left(\begin{array}{c}0\\\mb v\end{array}\right)\right\|_{\mbb X_1} = \|G_{12}(t)\mb v -G_{12}(s)\mb v\|_{1}+\|G_{22}(t)\mb v-G_{22}(s)\mb v\|_{1},
$$
so we see that both $G_{12}(t)$ and $G_{22}(t)$ are nonnegative (as we have coordinate-wise order) and bounded operators from $l^1(\mbb C^k)$ to, respectively,  $L^1((0,1), \mbb C^m)$ and $l^1(\mbb C^k),$ strongly continuous ($G_{22}$ even uniformly since it is a matrix), and satisfy the same exponential growth estimate as $G_{1}.$ Moreover, $G_{jj}(t), j=1,2,$ strongly converge to the identities in the respective spaces, and $G_{ij}(t), i\neq j,$ strongly converge to zero, as $t\to 0^+$.
Then \eqref{mild1} can be written as 
\begin{subequations}\label{mild2}
\begin{equation}
\mb u(t) = G_{11}(t)\mb{\mr u} + G_{12}(t)\mb{\mr v} + \int\limits_0^t G_{12}(t-s) \mc Q(\mb v(s))ds,\label{mild2a}
\end{equation}
\begin{equation}
\mb v(t) = G_{21}(t)\mb{\mr u} + G_{22}(t)\mb{\mr v} + \int\limits_0^t G_{22}(t-s) \mc Q(\mb v(s))ds,\label{mild2b}
\end{equation}
\end{subequations}
thus we have a hierarchical system, and it suffices to solve \eqref{mild2b} (which is a finite-dimensional system)  for $\mb v$ and substitute the solution into \eqref{mild2a}. 

Often, we are concerned with nonnegative solutions to \eqref{mod1} with nonnegative $\mc B,\mc F$ and $\mc G$, and Metzler $\mc A$, so that $G_{1}(t)\geq 0$, but with $\mc Q$ that is not necessarily nonnegative for $\mb v\geq 0.$ If, however, for any ball $B(0,R)$  there is $\alpha\geq 0$ such that
\begin{equation}
\mc Q(\mb v)+\alpha \mb v\geq 0, \quad \mb v\in B_+(0,R):=B(0,R)\cap \mbb R^k_+,
\label{poscon}
\end{equation}
then, we can consider an equivalent problem
\begin{equation}\label{mod3}
\begin{split}
\partial_t\left(\begin{array}{c}\mb u\\\mb v\end{array}\right) &= \left(\begin{array}{cc}-\mc C\mc{diag}\,\mb{\p}_x&\mb 0\\\mc B\gamma_1&\mc A-\alpha\mc I\end{array}\right)\left(\begin{array}{c}\mb u\\\mb v\end{array}\right) + \mbb Q_\alpha (\mb v)=: \mbb A_\alpha\left(\begin{array}{c}\mb u\\\mb v\end{array}\right) +  \mbb Q_\alpha (\mb v),
\end{split}
\end{equation}
where
$$
\mbb Q_{\alpha} (\mb v) = \left(\begin{array}{c} \mb 0\\\mc Q(\mb v)+\alpha \mb v\end{array}\right).
$$
Now, the semigroup generated by $\mbb{A}_\alpha$, defined on $D(\mbb A_\alpha)=D(\mbb A)$, is nonnegative by Proposition \ref{prop5} (as $s(\mc A-\alpha\mc I)<s(\mc A)),$ and 
$
G_{\alpha}(t)\leq G_{1}(t),
$
by \eqref{res1b} and an obvious estimate $R(\la,\mc A-\alpha\mc I)\leq R(\la, \mc A), \la >s(\mc A).$ 
Then, with $G_{ij}(t)$ replaced by $G_{\alpha,ij}(t)$, $i,j=1,2,$ \eqref{mild2b} takes the form  
\begin{equation}
\mb v(t) = G_{\alpha,21}(t)\mb{\mr u} + G_{\alpha,22}(t)\mb{\mr v} + \int\limits_0^t G_{\alpha,22}(t-s)(\mc Q(\mb v(s))+\alpha\mb v)ds.\label{mild3b}
\end{equation}
Thus, the Picard iterates of \eqref{mild3b} in $B_+(0,R)$ are nonnegative and hence $\mb u$ is nonnegative.
Thus, we can formulate
\begin{thm}
For any $(\mb {\mr u},\mb{\mr  v})\in \mbb X_{1,+}$, there exists a unique mild solution $(\mb u(t),\mb v(t))$ to \eqref{mod2} defined on its maximal interval of existence $[0,t_{\max}(\mb {\mr u},\mb{\mr  v})),$ which is nonnegative if \eqref{poscon} is satisfied. 
The solution is classical if $\mc Q$ is differentiable and $(\mb {\mr u},\mb{\mr  v})\in D(\mbb A).$ If $t_{\max}(\mb {\mr u},\mb{\mr  v})<\infty,$ then $\lim\limits_{t\to t_{\max}(\mb {\mr u},\mb{\mr  v})}\|(\mb u(t),\mb v(t))\|_{\mbb X_1} =\infty.$
\end{thm}
The proof of this result is standard, e.g., \cite[Theorems 6.1.2\,\& 6.1.5]{Pa}.

\subsection{Stability analysis}

We assume that $\mc Q$ is differentiable. To get the stationary points of \eqref{mod3}, we first observe that \eqref{mod1a} implies that $\mb u(x)=\mb u^*$ is a constant. Then, $\gamma_0\mb u^* = \gamma_1\mb u^*=\mb u^*$, which yields
\begin{equation}
\begin{split}
\mc A\mb v^* + \mc Q(\mb v^*) +\mc B\mb u^*&= 0,\\
\mc F\mb v^* + \mc G\mb u^*&=\mb u^*.
\end{split}
\label{stp1}
\end{equation}
To get the linearisation around $(\mb u^*,\mb v^*)$, we define $\mb z=\mb u-\mb u^*, \mb y = \mb v-\mb v^*$. Then, $\mb z\in (W^1_1(0,1),\mbb R^m)$ if $\mb u$ does, and
$$
\gamma_0 \mb z = \gamma_0\mb u - \mb u^* = \mc F\mb v+\mc G\gamma_1\mb u-\mb u^* = \mc F\mb y + \mc F\mb v^* + \mc G\gamma_1\mb z + \mc G\mb u^* -\mb u^*= \mc F\mb y +\mc G\gamma_1\mb z
$$
on account of the second formula in \eqref{stp1}. Hence, $(\mb u,\mb v)\in D(\mbb A)$ implies $(\mb z,\mb y)\in D(\mbb A)$ and the linearised systems is given by 
\begin{equation}
\begin{split}
    \p_t \mb z &= -\p_x \mb z,\\
    \mb y'&= \mc A\mb y +\mc J_{\mc Q}(\mb u^*,\mb v^*)\mb y +\mc B\gamma_1\mb z,\\
    \gamma_0\mb z&= \mc F\mb y + \mc G\gamma_1\mb z,
\end{split}
\label{lin1}
\end{equation}
where $\mc J_{\mc Q}(\mb u^*,\mb v^*)$ is the Jacobian of $\mc Q$ at $(\mb u^*,\mb v^*).$ Denote by $\sem {*}$ the $C_0$-semigroup generated by \eqref{lin1} and let $\omega^*$ be its type. 

Then, e.g., \cite[Proposition 4.17]{Webb} implies
\begin{thm}
If $\mc Q$ is differentiable and $\omega^*<0$,  then the stationary solution $(\mb u^*,\mb v^*)$ is locally asymptotically stable in the sense that there are $\e>0,\gamma >0, M\geq 1$ such that for any $(\mb {\mr u},\mb {\mr v})\in \mbb X_1$ satisfying  $\|(\mb {\mr u},\mb {\mr v})-(\mb u^*,\mb v^*)\|_{\mbb X_1}<\e$, the mild solution $(\mb u(t),\mb v(t))$ to \eqref{mod2} is globally defined and $\|(\mb { u}(t),\mb { v}(t))-(\mb u^*,\mb v^*)\|_{\mbb X_1}\leq Me^{-\gamma t}\|(\mb {\mr u},\mb {\mr v})-(\mb u^*,\mb v^*)\|_{\mbb X_1}$ for all $t\geq 0.$
\end{thm}

\subsubsection{Example: an $SIS$ model on a network}
For illustrative purposes, we consider a multi-patch $SIS$ disease with no vertical disease transmission. We assume that the infectives do not travel so that the disease develops only in the cities represented by the nodes of a digraph. Then, each population $v_l$ at site $ \nu_l$, $l\in I_v,$ is divided into susceptibles $s_l$ and infectives $i_l$ and the population without migrations evolves as   \begin{equation*}
\begin{split}
s_l'&= \Lambda_l -\mu^s_l s_l -  \beta_ls_li_l +\omega_l i_l, \\
i_l'&=  \beta_ls_li_l -\mu^i_l i_l-\omega_l i_l,
\end{split}
\end{equation*}
$l\in I_v$. Here,
\begin{itemize}
    \item the force of infection is given by the mass-action law $\beta_ls_li_l$,
    \item the demography is affine with the total birth rate $\Lambda_l$,
    \item $\mu^s_l$ and $\mu^i_l$ are the disease-state specific death rates per capita, 
    \item $\omega_l$ is the recovery rate (without immunity).
\end{itemize}
  
  Next, let the sites $\nu_l$ be the vertices of a digraph $\mbb G$ with directed edges (arcs) $\mb e_j, j\in I_e$, along which susceptibles $u_j$ move with some speed $c_j$. Then, at time $t$, at $ \nu_l,$ we will have the resident susceptibles $s_l(t)$ and the migrants arriving at the rates $b_{lj} = \phi^+_{lj}, j\in E^+_l,$ see Section \ref{secmd}. The latter's role in the infection is negligible, as can be seen from the following simple reasoning. Discarding the demography and recovery, the change in the number of susceptibles over a short period of time $\Delta t$ equals
  $$
  s_l(t+\Delta t)-s_l(t) = -\beta_l(s_l(t+\delta )+s^m_l(t+\delta))i_l(t+\delta)\Delta t, \quad 0\leq \delta \leq \Delta t,
  $$
  where $s^m_l(t+\delta)$ are new susceptibles that arrived between $t$ and $t+\delta$. The latter will be smaller than the number of new arrivals over the period $[t,t+\Delta t].$  Thus, if the new arrivals occur at some finite rate,  then $s^m_l(t+\delta) = O(\Delta t)$ and hence the contribution of the newly arriving susceptibles to the infection rate will be of $(\Delta t)^2$ order and thus will vanish when we pass to the derivatives. Thus, as in, e.g., \cite{arino2009}, 
    \begin{equation*}
\begin{split}
s_l'&= \Lambda_l -\mu^s_l s_l -  \beta_l s_l i_l +\omega_l i_l -\mc m_{ll} s_l + \sum_{j\in I_e} b_{lj}\gamma_1u_j, \\
i_l'&=  \beta_l s_l i_l -\mu^i_l i_l-\omega_l i_l.
\end{split}
\end{equation*} 
To combine this model with \eqref{de1}, for $\mb s = (s_1,\ldots,s_k), \mb i = (i_1,\ldots,i_k)$, we define
$$
\mb v = \left(\begin{array}{c}\mb s\\\mb i\end{array}\right).
$$
Then, \eqref{de1b} and \eqref{de1c} are modified to 
\begin{equation*}
\begin{split}
\mb v_t &= -\mc M_d\mb v+ \mc Q(\mb v) + \mc B \mb u(t,1),\\
\mc C\mb u(t,0)& = \mc F \mb v(t) + \mc G\mb u(t,1),
\end{split}
\end{equation*}
where 
$$
\mc A = \left(\begin{array}{cc}-\mc M_d&\mb 0\\
\mb 0&\mb 0\end{array}\right), \quad 
\mc Q(\mb v) = \left(\begin{array}{c} \Lambda_1 -\mu^s_1 s_1 -  \beta_1 s_1 i_1 +\omega_1 i_1\\
\vdots\\
\Lambda_k -\mu_k^s s_k -  \beta_k s_k i_k +\omega_k i_k\\
\beta_1 s_1 i_1 -\mu^i_1 i_1-\omega_1 i_1\\
\vdots\\
\beta_k s_k i_k -\mu^i_k i_k-\omega_k i_k
\end{array}\right),
$$
and
$$
\mc B = \left(\begin{array}{c} \Phi_p^+ \mc C \\
\mb 0\end{array} \right), \quad 
\mc F = \left(\begin{array}{cc}  \Phi^-_{w_1}\mc M_d  &\mb 0\end{array} \right), \quad \mc G = \Phi^-_{w_2}\Phi^+_q \mc C.
$$
We assume that there are no sinks in $\mbb G,$ that is, $-\mc M_d$ is a stable Metzler matrix. 

The stationary states $(\mb u^*,\mb s^*,\mb i^*)$ are solutions to the relevant version of \eqref{stp1}, 
\begin{equation*}
\begin{split}
-\mc M_d\mb s^* + \mb \Lambda -(\operatorname{diag}\,\mb \mu^s )\mb s^* -(\operatorname{diag}\,(\mb\beta\odot \mb i^*))\mb s^* +  (\operatorname{diag}\,\mb \omega)\mb i^*+\Phi_p^+\mc C \mb u^*&=\mb 0,\\
(\operatorname{diag}\,(\mb\beta\odot \mb i^*))\mb s^* - (\operatorname{diag}\,\mb \mu^i )\mb i^*  -(\operatorname{diag}\,\mb \omega)\mb i^*&=\mb 0,\\
\mc C\mb u^*-\Phi^-_{w_1}\mc M_d\mb s^* - \Phi^-_{w_2}\Phi^+_q \mc C\mb u^*&=\mb 0,
\end{split}
\end{equation*}
where $\odot$ denotes the Hadamard product of vectors. If we focus on a disease-free equilibrium (DFE), $(\mb u^*,\mb s^*,\mb 0)$, then we obtain the linear system
\begin{equation*}
\begin{split}
-\mc M_d\mb s^* + \mb \Lambda -(\operatorname{diag}\,\mb \mu^s )\mb s^* +\Phi_p^+\mc C \mb u^*&=\mb 0,\\
\mc C\mb u^*-\Phi^-_{w_1}\mc M_d\mb s^* - \Phi^-_{w_2}\Phi^+_q \mc C\mb u^*&=\mb 0.
\end{split}
\end{equation*}
Assuming that $r(\Phi^-_{w_2}\Phi^+_q)<1,$ see Lemma  \ref{lem2}, we arrive at
\begin{equation}
\mc T\mb s^*:= -(\mc M_d+\operatorname{diag}\,\mb \mu^s )\mb s^* +\Phi_p^+(\mc I - \Phi^-_{w_2}\Phi^+_q)^{-1}\Phi_{w_1}^-\mc M_d\mb s^*=-\mb \Lambda. 
\label{Teq}
\end{equation}
Using again \cite[Theorems 8.1 \& 8.5]{banpop}, $s(\mc T)<0$ (ensuring the positive invertibility of $-\mc T$), if and only if $r(\Phi_p^+(\mc I - \Phi^-_{w_2}\Phi^+_q)^{-1}\Phi_{w_1}^-\mc M_d(\mc M_d+\operatorname{diag}\,\mb \mu^s )^{-1})<1.$ However,
$$
\Phi_p^+(\mc I - \Phi^-_{w_2}\Phi^+_q)^{-1}\Phi_{w_1}^-\mc M_d(\mc M_d+\operatorname{diag}\,\mb \mu^s )^{-1}\leq \max\limits_{i\in I_v}\frac{\mc m_{ii}}{\mc m_{ii} + \mu^s_i}\Phi_p^+(\mc I - \Phi^-_{w_2}\Phi^+_q)^{-1}\Phi_{w_1}^-.
$$
Hence, using Theorem \ref{thm5} and Remark \ref{rem7}, there is a unique positive solution to \eqref{Teq}, $\mb s^*=-\mc T^{-1}\mb \Lambda$ if $\min\limits_{i\in I_v} \mu^s_i>0$ and $p_{ir}+q_{rj}\leq 1$ for any $i,j\in E^+_r, r\in I_v.$ This occurs, in particular, when at the endpoint of each arc, there is a perfect split of the population into the transient part and the part entering the site, and there is nonzero mortality at each site. 

For the local asymptotic stability of $\mb v^* = (\mb s^*,\mb 0),$  we evaluate the Jacobian of $\mc Q$
 at the DFE,  
\begin{align*}
\mc J_\mc Q(\mb v^*)& = \left(\begin{array}{cc}-\operatorname{diag}\;\mb \mu^s&\operatorname{diag} (-\mb\beta\odot\mb s^*+\mb\omega)\\
\mb 0& \operatorname{diag}( \mb\beta\odot\mb s^*-\mb\mu^i-\mb \omega)\end{array}\right). 
\end{align*}
We see that, in general, $\mc J_\mc Q(\mb v^*)$ is not a Metzler matrix. However, using Theorem \ref{mstth1}, we consider 
\begin{align*}
\mc J^+_\mc Q(\mb v^*) &= \left(\begin{array}{cc}-\operatorname{diag}\;\mb \mu^s&\operatorname{diag} |-\mb\beta\odot\mb s^*+\mb\omega|\\
\mb 0& \operatorname{diag}( \mb\beta\odot\mb s^*-\mb\mu^i-\mb \omega)\end{array}\right) 
\end{align*}
and we see that the DFE is locally asymptotically stable if 
\begin{align*}
&s\left(\left(\begin{array}{cc}\mc T&\operatorname{diag} |-\mb\beta\odot\mb s^*+\mb\omega|\\
\mb 0&\operatorname{diag}( \mb\beta\odot\mb s^*-\mb\mu^i-\mb \omega)\end{array}\right)
\right) \\
&=
s\left(\left(\begin{array}{cc}-\mc M_d-\operatorname{diag}\,\mb \mu^s+\Phi_p^+(\mc J-\Phi_{w_2}^-\Phi^+_q)^{-1}\Phi_{w_1}^-\mc M_d&\operatorname{diag} |-\mb\beta\odot\mb s^*+\mb\omega|\\
\mb 0&\operatorname{diag}( \mb\beta\odot\mb s^*-\mb\mu^i-\mb \omega)\end{array}\right)
\right)<0.
\end{align*}
By the assumptions ensuring the existence of a positive $\mb s^*$, $s(\mc T)<0.$ Hence, the DFE is locally asymptotically stable provided 
$$
\frac{\beta_rs_r^*}{\mu_r^i+\omega_r}<1, \quad r\in I_v,
$$
that is, we can define the basic reproduction number
\begin{equation*}
     R_0 = \max\limits_{r\in I_v}\frac{\beta_rs_r^*}{\mu_r^i+\omega_r}.
\end{equation*}
The structure of $R_0$ is typical for SIS models with migration of only susceptible individuals, see, e.g., \cite[Eqn (3.15)]{SalvdD} or \cite[Theorem 3.8]{arino2009}, adjusted to the mass-action law for the force of infections, but the difference lies in $\mb s^*$, which encodes the migration patterns. 
\section{Conclusions}
In this paper, we studied a coupled system of linear transport equations and nonlinear ordinary differential equations. By careful estimates of the resolvent of the linear part of the problem, we proved the existence of a $C_0$-semigroup solving it in $L^1$, and, by showing that the $L^1$ semigroup leaves $L^p$ spaces $(1<p<\infty)$ invariant, we extended the well-posedness to them. Next, we proved several results on the asymptotics of the obtained semigroup when it is positive, and extended them to arbitrary semigroups. The latter played an essential role in the study of the stability of the stationary points of the full nonlinear version of \eqref{mod1}, where, typically, the Jacobian in the linearisation of the problem has entries of arbitrary sign. We illustrated our results on an epidemiological example. 

Further work on this problem involves the existence of a stationary distribution in the conservative linear case and the related question of whether the problem has the Asynchronous Exponential Growth property, i.e., a spectral gap. Another avenue of research is the extension of the problem to include nonlinearities in the transport equations and boundary operators, and, related to this, the study of more complex ecological, biological, and epidemiological problems in metapopulations. 
\section{Declarations  and acknowledgments}
\begin{enumerate}
\item \textbf{Funding.} The research was supported by the National Research Foundation of South Africa, grant 87720, and by the IMPRESS-U programme in the framework of the project
\#2023/05/Y/ST6/00263.
\item \textbf{Conflicts of Interest.} The authors declare that they have no conflict 
of interest.
\item \textbf{Ethical approval.} Not applicable.
\item \textbf{Data availability.} Data sharing not applicable to this article as no
datasets were generated or analysed during the current study.
\item \textbf{Acknowledgment.} The work began at the Workshop \& Training School of the Networking in Applied Network Theory (NANT) at the Mathematical Research and Conference Centre in B\c{e}dlewo, Poland, and concluded during J.Banasiak's fellowship at the Stellenbosch Institute for Advanced Studies. 
\end{enumerate}
\bibliographystyle{plain}
\def\cprime{$'$} \def\cprime{$'$} \def\cprime{$'$} \def\cprime{$'$}
  \def\cprime{$'$} \def\cprime{$'$} \def\cprime{$'$} \def\cprime{$'$}
  \def\cprime{$'$} \def\cprime{$'$}


\begin{thebibliography}{10}

\bibitem{Adams}
R.~A. Adams and J.~J.~F. Fournier.
\newblock {\em Sobolev spaces}, volume 140 of {\em Pure Appl. Math., Academic
  Press}.
\newblock New York, NY: Academic Press, 2nd ed. edition, 2003.

\bibitem{Arendt}
W.~Arendt.
\newblock Resolvent positive operators.
\newblock {\em Proc. London Math. Soc. (3)}, 54(2):321--349, 1987.

\bibitem{arino2009}
J.~Arino.
\newblock Diseases in metapopulations.
\newblock In {\em Modeling and dynamics of infectious diseases}, pages 64--122.
  World Scientific, 2009.

\bibitem{BanLAA}
J.~Banasiak.
\newblock Explicit formulae for limit periodic flows on networks.
\newblock {\em Linear Algebra Appl.}, 500:30--42, 2016.

\bibitem{banpop}
J.~Banasiak.
\newblock {\em Introduction to mathematical methods in population theory}.
\newblock Springer Undergrad. Math. Ser. Cham: Springer, 2025.

\bibitem{JBAB1}
J.~Banasiak and A.~B{\l}och.
\newblock Telegraph systems on networks and port-{Hamiltonians}. {I}.
  {Boundary} conditions and well-posedness.
\newblock {\em Evol. Equ. Control Theory}, 11(4):1331--1355, 2022.

\bibitem{BaFa}
J.~Banasiak and A.~Falkiewicz.
\newblock Some transport and diffusion processes on networks and their graph
  realizability.
\newblock {\em Appl. Math. Lett.}, 45:25--30, 2015.

\bibitem{BanRot}
J.~Banasiak and A.~Falkiewicz.
\newblock A singular limit for an age structured mutation problem.
\newblock {\em Math. Biosci. Eng.}, 14(1):17--30, 2017.

\bibitem{BFNM3AS}
J.~Banasiak, A.~Falkiewicz, and P.~Namayanja.
\newblock Asymptotic state lumping in transport and diffusion problems on
  networks with applications to population problems.
\newblock {\em Math. Models Methods Appl. Sci.}, 26(2):215--247, 2016.

\bibitem{BNNHM}
J.~Banasiak and P.~Namayanja.
\newblock Asymptotic behaviour of flows on reducible networks.
\newblock {\em Netw. Heterog. Media}, 9(2):197--216, 2014.

\bibitem{BPChapt}
J.~Banasiak and A.~Puchalska.
\newblock Transport on networks -- a playground of continuous and discrete
  mathematics in population dynamics.
\newblock In {\em Mathematics applied to engineering, modelling, and social
  issues}, pages 439--487. Cham: Springer, 2019.

\bibitem{Bang}
J.~Bang-Jensen and G.~Gutin.
\newblock {\em Digraphs. {Theory}, algorithms and applications.}
\newblock Springer Monogr. Math. London: Springer, 2nd ed. edition, 2009.

\bibitem{BaCo}
G.~Bastin and J.-M. Coron.
\newblock {\em Stability and boundary stabilization of 1-{D} hyperbolic
  systems}, volume~88 of {\em Prog. Nonlinear Differ. Equ. Appl.}
\newblock Basel: Birkh{\"a}user/Springer, 2016.

\bibitem{batkai2017positive}
A.~B{\'a}tkai, M.~Kramar~Fijav{\v{z}}, and Abdelaziz Rhandi.
\newblock {\em Positive Operator Semigroups: from Finite to Infinite
  Dimensions}, volume 257.
\newblock Birkh{\"a}user, 2017.

\bibitem{BatPia}
A.~B{\'a}tkai and S.~Piazzera.
\newblock {\em Semigroups for delay equations}, volume~10 of {\em Res. Notes
  Math.}
\newblock Wellesley, MA: A K Peters, 2005.

\bibitem{Belair}
J.~B{\'e}lair, M.~C. Mackey, and J.~M. Mahaffy.
\newblock Age-structured and two-delay models for erythropoiesis.
\newblock {\em Math. Biosci.}, 128(1-2):317--346, 1995.

\bibitem{BatKraRh}
A.~Bátkai, M.~Kramar~Fijavž, and A.~Rhandi.
\newblock Abstract boundary delay systems and application to network flow.
\newblock {\em Mathematical Methods in the Applied Sciences}, 49(1):119--129,
  2026.

\bibitem{Copp}
W.~A. Coppel.
\newblock {\em Stability and asymptotic behavior of differential equations}.
\newblock D. C. Heath and Company, Boston, MA, 1965.

\bibitem{KraPhysD}
B.~Dorn, M.~Kramar~Fijav{\v{z}}, R.~Nagel, and A.~Radl.
\newblock The semigroup approach to transport processes in networks.
\newblock {\em Physica D}, 239(15):1416--1421, 2010.

\bibitem{DoumicMC}
M.~Doumic, A.~Marciniak-Czochra, B.~Perthame, and J.~P. Zubelli.
\newblock A structured population model of cell differentiation.
\newblock {\em SIAM J. Appl. Math.}, 71(6):1918--1940, 2011.

\bibitem{ENshort}
K.-J. Engel and R.~Nagel.
\newblock {\em A short course on operator semigroups}.
\newblock Universitext. Springer, New York, 2006.

\bibitem{Gant}
F.~R. Gantmacher.
\newblock Applications of the theory of matrices.
\newblock New {York}-{London}: {Interscience} {Publishers}. {IX}, 317 p.
  (1959)., 1959.

\bibitem{HaddRhandi}
S.~Hadd, R.~Manzo, and A.~Rhandi.
\newblock Unbounded perturbations of the generator domain.
\newblock {\em Discrete Contin. Dyn. Syst.}, 35(2):703--723, 2015.

\bibitem{bein}
R.~L. Hemminger and L.~W. Beineke.
\newblock Line graphs and line digraphs.
\newblock Selected topics in graph theory, 271-305 (1978)., 1978.

\bibitem{KamSal}
J.~C. Kamgang and G.~Sallet.
\newblock Computation of threshold conditions for epidemiological models and
  global stability of the disease-free equilibrium ({DFE}).
\newblock {\em Math. Biosci.}, 213(1):1--12, 2008.

\bibitem{KSMZ}
M.~Kramar and E.~Sikolya.
\newblock Spectral properties and asymptotic periodicity of flows in networks.
\newblock {\em Math. Z.}, 249(1):139--162, 2005.

\bibitem{KrPuch}
M.~Kramar~Fijav{\v{z}} and A.~Puchalska.
\newblock Semigroups for dynamical processes on metric graphs.
\newblock {\em Philos. Trans. R. Soc. Lond., A, Math. Phys. Eng. Sci.},
  378(2185):17, 2020.
\newblock Id/No 20190619.

\bibitem{Matrai}
T.~M{\'a}trai and E.~Sikola.
\newblock Asymptotic behavior of flows in networks.
\newblock {\em Forum Math.}, 19(3):429--461, 2007.

\bibitem{Mugnolo}
D.~Mugnolo.
\newblock {\em Semigroup methods for evolution equations on networks}.
\newblock Underst. Complex Syst. Cham: Springer, 2014.

\bibitem{Nic}
G.~Nickel.
\newblock A semigroup approach to dynamic boundary value problems.
\newblock {\em Semigroup Forum}, 69(2):159--183, 2004.

\bibitem{Ortega}
J.~M. Ortega.
\newblock {\em Numerical analysis}, volume~3 of {\em Classics in Applied
  Mathematics}.
\newblock Society for Industrial and Applied Mathematics (SIAM), Philadelphia,
  PA, second edition, 1990.

\bibitem{Pa}
A.~Pazy.
\newblock {\em Semigroups of linear operators and applications to partial
  differential equations}, volume~44 of {\em Applied Mathematical Sciences}.
\newblock Springer-Verlag, New York, 1983.

\bibitem{Pucharx}
A.~Puchalska, M.~N.~Cartier van Dissel, P.~Gora, M.~Iskrzy{\'n}ski, M.~Kramar
  Fijav{\v{z}}, D.~Manea, A.~Mauroy, I.~Naki{\'c}, S.~Nicaise, M.~B.
  Paradowski, G.~Rotundo, and E.~Sikolya.
\newblock Journey {Through} the {World} of {Dynamical} {Systems} on {Networks}.
\newblock Preprint, {arXiv}:2512.17571 [math.{DS}] (2025), 2025.

\bibitem{Rot}
M.~Rotenberg.
\newblock Transport theory for growing cell populations.
\newblock {\em J. Theoret. Biol.}, 103(2):181--199, 1983.

\bibitem{SalvdD}
M.~Salmani and P.~van~den Driessche.
\newblock A model for disease transmission in a patchy environment.
\newblock {\em Discrete Contin. Dyn. Syst., Ser. B}, 6(1):185--202, 2006.

\bibitem{Sik}
E.~Sikolya.
\newblock Flows in networks with dynamic ramification nodes.
\newblock {\em J. Evol. Equ.}, 5(3):441--463, 2005.

\bibitem{Staff}
O.~J. Staffans.
\newblock {\em Well-posed linear systems}, volume 103 of {\em Encycl. Math.
  Appl.}
\newblock Cambridge: Cambridge University Press, 2005.

\bibitem{Varga}
R.~S. Varga.
\newblock {\em Matrix iterative analysis.}, volume~27 of {\em Springer Ser.
  Comput. Math.}
\newblock Berlin: Springer, 2nd revised and expanded edition, 2000.

\bibitem{Webb}
G.~F. Webb.
\newblock {\em Theory of nonlinear age-dependent population dynamics},
  volume~89 of {\em Pure Appl. Math., Marcel Dekker}.
\newblock Marcel Dekker, Inc., New York, NY, 1985.

\bibitem{Weiss}
G.~Weiss.
\newblock Regular linear systems with feedback.
\newblock {\em Math. Control Signals Syst.}, 7(1):23--57, 1994.

\bibitem{Zwart}
H.~Zwart, Y.~Le~Gorrec, B.~Maschke, and J.~Villegas.
\newblock Well-posedness and regularity of hyperbolic boundary control systems
  on a one-dimensional spatial domain.
\newblock {\em ESAIM, Control Optim. Calc. Var.}, 16(4):1077--1093, 2010.

\end{thebibliography}
\end{document}